\documentclass[a4paper]{cas-sc}

\usepackage{graphicx}
\usepackage{amsmath,amssymb,amsfonts}
\usepackage{amsthm}
\usepackage[title]{appendix}
\usepackage{mathtools}
\usepackage[numbers,sort&compress]{natbib}

\theoremstyle{plain}
\newtheorem{theorem}{Theorem}
\newtheorem{lemma}{Lemma}
\newtheorem{corollary}{Corollary}

\theoremstyle{definition}
\newtheorem{definition}{Definition}

\theoremstyle{remark}
\newtheorem{remark}{Remark}

\DeclareMathOperator{\val}{clearance}

\begin{document}
\ExplSyntaxOn
\cs_set:Npn \__first_footerline: {\mbox{}}
\ExplSyntaxOff

	\let\WriteBookmarks\relax
	\def\floatpagepagefraction{1}
	\def\textpagefraction{.001}
	
	\shorttitle{Realizability of Kaleidocycles}
	\shortauthors{H. Kociemba}
	
	\title[mode=title]{Necessary and Sufficient Criteria for the Realizability of Kaleidocycles Composed of Six Congruent Isosceles Tetrahedra}
	
	\author[1]{Herbert Kociemba}[orcid=0009-0006-1477-974X]
	\ead{kociemba@t-online.de}
	
	\affiliation[1]{
		organization={TU Darmstadt},
		city={Darmstadt},
		country={Germany}
	}

\begin{abstract}
	A classical kaleidocycle is a closed cyclic chain of congruent,
	edge-hinged tetrahedra with successive hinge axes at right angles,
	capable of continuous inside-out motion. Threefold-symmetric kaleidocycles composed of six congruent isosceles tetrahedra (disphenoids) form a natural generalization in which consecutive hinge axes need not be perpendicular. Their kinematics
	are governed by the closure equation for threefold-symmetric Bricard 6R
	linkages.
	
	We derive a compact half-angle form of the closure equation directly from threefold symmetry using the composition of rotations, and establish its equivalence to the classical formulation. Furthermore, we establish three equivalent necessary and sufficient criteria for a kaleidocycle to admit continuous, unblocked motion in terms of the geometric parameters of the constituent disphenoid. These criteria yield sharp bounds on the admissible geometric parameters of realizable kaleidocycles.
	
	Finally, we study assemblies of eight such kaleidocycles joined
	face-to-face. We identify a configuration exhibiting full octahedral
	symmetry, which forms a polyhedral structure closely related to the
	disdyakis dodecahedron. Although the disphenoids fit together without gaps or overlaps in this configuration, the rigid-body model predicts slight overlaps between adjacent disphenoids during the motion of the individual kaleidocycles.

\end{abstract}

	
	\maketitle
	

\section{Introduction}
 A classical kaleidocycle consists of six congruent tetrahedra connected edge-to-edge in a cyclic chain and admits a continuous rotating motion. Paul Schatz invented such a mechanism in 1929 (see description in~\cite{schatz1998}). This mechanism is now commonly referred to as the “Schatz cube” or “invertible cube” (Figure~\ref{fig:schatzcube}). In 1935, Stalker~\cite{Stalker35} independently developed a toy with the same principle. Public awareness of kaleidocycles began only with the work of Schattschneider and Walker~\cite{schattschneider1977} from 1977 onward. They refined Stalker’s earlier design -- also based on isosceles triangular faces -- and further enhanced it by decorating the surfaces with M. C. Escher’s periodic patterns. Walker also coined the term ``kaleidocycle'' (Greek: \textit{kalós} [beautiful] + \textit{eîdos} [form] + \textit{kýklos} [ring]).

All these models possess point group symmetry $C_{3v}$ throughout their
motion, with a threefold rotational axis and three vertical mirror
planes. Furthermore, the two opposite edges of each tetrahedron used to connect it to its two neighbors are perpendicular.

Essential aspects of the kinematics of a kaleidocycle can be revealed
when we set aside its geometric shape and concentrate on the underlying
kinematic linkage formed by rigid links (bars) and revolute joints
(hinges).
Mechanisms closely related to kaleidocycles were first studied by
Bricard in 1927~\cite{Bricard1927}. These are now known as Bricard 6R
linkages. They are of particular interest because they provide rare
examples of spatial mechanisms that, according to the
Chebychev--Gr\"ubler--Kutzbach mobility
criterion~\cite{uicker2011theory}, should be rigid, yet nevertheless
admit continuous motion.

In linkage terminology, classical kaleidocycles are closed-loop kinematic chains (or ring linkages) consisting of six links and six hinges, in which adjacent hinge axes are mutually perpendicular.  
\begin{figure}
	\centering
	\begin{minipage}[t]{0.3\textwidth}
		\includegraphics[width=\textwidth]{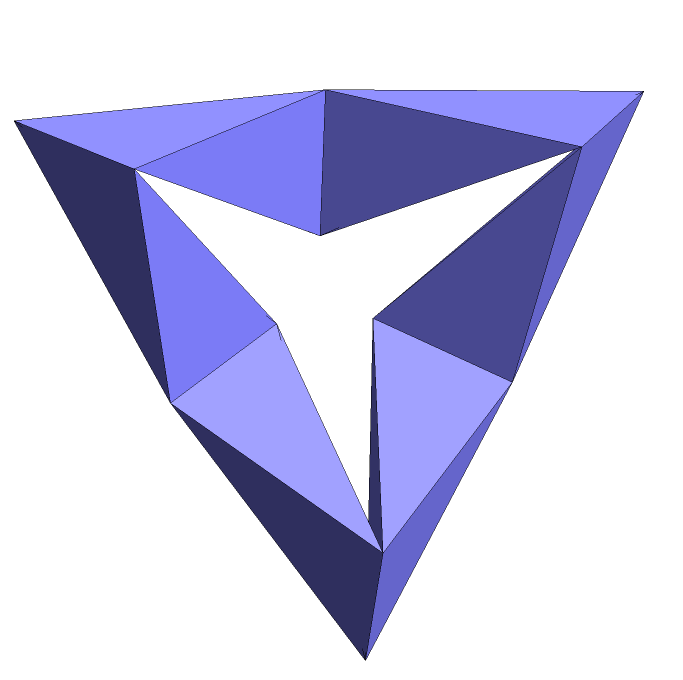}
	\end{minipage}
	\begin{minipage}[t]{0.3\textwidth}
		\includegraphics[width=\textwidth]{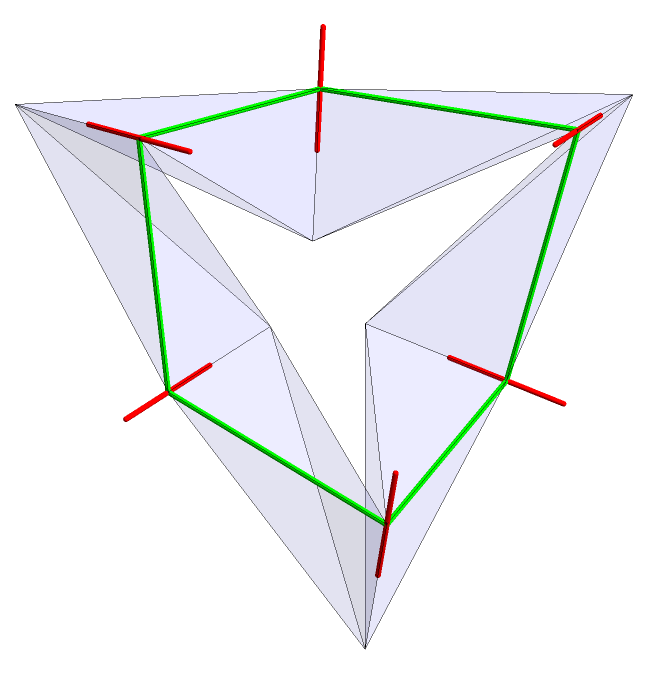}
	\end{minipage}
	\caption{The Schatz cube and its corresponding linkage. The hinges at both ends of a bar are perpendicular to each other.}
	\label{fig:schatzcube}
\end{figure}

Bricard also investigated 6R linkages with non-perpendicular adjacent hinge axes, and their closure equations were subsequently analyzed by Baker~\cite{Baker1980}. The particular class characterized by threefold rotational symmetry was studied by Gan and Pellegrino~\cite{gan2003closed} and Chen et al.~\cite{chen2005threefold}. These mechanisms are known as \emph{threefold-symmetric Bricard 6R linkages}. Mechanisms based on this design are used in deployable structures, such as space antennas and solar wings, as well as in architecture (see, for example,~\cite{luo2008retractable}).

Another generalization of the classical kaleidocycle is the class of Möbius kaleidocycles. The invention was documented in a Japanese patent application filed in 2018 by Schönke, Kaji, Fried, and Grunwald~\cite{moebiusPatent2019}, and subsequently described in a mathematical paper by Schönke and Fried~\cite{Schoenke2019}. Like the threefold-symmetric Bricard 6R linkages considered here, they are constructed from cyclic chains of congruent isosceles tetrahedra, but they differ in their symmetry and geometric structure. In this paper, we focus exclusively on the threefold-symmetric class introduced by Gan and Pellegrino~\cite{gan2003closed}. Further details on Möbius kaleidocycles can be found in~\cite{boni2025kinematic,husty2022moebius,kaji2020linkage,kaji2023explicit}.

\medskip

\paragraph{Contributions.}
The threefold-symmetric Bricard 6R linkage and its closure equation were studied by Gan and Pellegrino~\cite{gan2003closed} and Chen et al.~\cite{chen2005threefold}. Chen et al. also analyzed the compatibility paths and their bifurcations, including the physical blocking of certain bar-linkage configurations. The hinge-angle relation used in this paper was previously derived by Safsten et al.~\cite{safsten2016}. 
While the basic geometry of disphenoids is classical~\cite{MR38667}, Section~\ref{sec:tetraformulas} derives explicit formulas for their dihedral angles, including a compact expression involving only the three edge lengths.

The main contribution of this paper is the derivation of necessary and sufficient conditions for the realizability of threefold-symmetric kaleidocycles in terms of the geometric parameters of the constituent disphenoid (Theorems~\ref{theorem:nonblocking}--\ref{theorem:noninterference2}).
Here, a kaleidocycle is called realizable if its constituent tetrahedra can be assembled into a closed ring and undergo a complete cycle of motion without their interiors intersecting.

A key ingredient is a compact half-angle reformulation of the known closure equation for the underlying threefold-symmetric Bricard 6R linkage,
\[
\cos\frac{\theta}{2}\cos\frac{\varphi}{2}
-\cos\alpha\sin\frac{\theta}{2}\sin\frac{\varphi}{2}
=\pm\frac12,
\]
derived in Theorem~\ref{theorem:phi_theta_short} from threefold symmetry using the composition of rotations. We show that this equation is equivalent to the classical closure equation and allows a simpler derivation of the hinge-angle relations previously obtained by Safsten et al.~\cite{safsten2016}, including the explicit relation between the hinge angles and their admissible ranges (Theorems~\ref{theorem:phi_theta} and~\ref{theorem:anglerange}). These relations are in turn essential for proving our realizability conditions.

The realizability criteria also yield results concerning the existence of kaleidocycles with special or extremal properties. After scaling the hinge edge to length 1, we show that the minimal possible length of a non-hinge edge of the disphenoid in a realizable kaleidocycle is a constant expressible in radicals, approximately 1.10609. These criteria further yield bounds on the edge angles of the isosceles tetrahedra: the angle between the hinge edge and an adjacent edge is at least $60^\circ$, while the angle opposite the hinge edge is at most $\arccos(3/5)\approx 53^\circ$.

The only kaleidocycle attaining the extremal angle $\arccos(3/5)$ is the classical kaleidocycle of Walker. This kaleidocycle also has another extremal property: the dihedral angle at the hinge edge is $60^\circ$, whereas for all other realizable kaleidocycles it is strictly smaller.

Finally, motivated by physical realizations of kaleidocycles assembled from six isosceles tetrahedra, we investigate whether eight congruent kaleidocycles can be joined face-to-face to form a rigid polyhedron. Unlike the mobile assemblies of threefold-symmetric Bricard linkages investigated by Zhang et al.~\cite{zhang2024assemblies}, our construction is based on face-to-face contacts between the constituent disphenoids. We construct a disphenoid that permits such an assembly of eight kaleidocycles into a polyhedral structure with full octahedral symmetry, closely related to the \emph{disdyakis dodecahedron}. 
A mathematical analysis nevertheless shows that the individual kaleidocycles cannot complete a full cycle of motion without blocking. Despite this theoretical obstruction, physical models exhibit no observable blocking during rotation; see Section~\ref{sec:doecahedron}.

\medskip
\textbf{Structure of this paper.}
In Section~\ref{sec:tetraformulas}, we derive several geometric properties of isosceles tetrahedra that are used throughout the paper. Section~\ref{sec:kinematicslinkage} studies the kinematics of the underlying
threefold-symmetric Bricard 6R linkage, 
including the derivation of a compact half-angle reformulation of the known closure equation and its equivalence to the classical formulation. In Section~\ref{sec:kinematicskaleido}, we derive the fundamental kinematic relations governing the motion of the corresponding kaleidocycles. Section~\ref{sec:realizable} establishes necessary and sufficient conditions for their realizability and investigates the resulting extremal geometric properties. Finally, Section~\ref{sec:doecahedron} studies highly symmetric face-to-face assemblies of eight congruent kaleidocycles.

In the remainder of the paper, the term \emph{kaleidocycle} always refers to the threefold-symmetric type, and all tetrahedra are assumed to be isosceles. Symbolic computations supporting the derivations, together with independent checks of the analytical results, including exact quantifier elimination over the real numbers, are documented in the supplementary Mathematica notebook.

 \section{Formulas for Isosceles Tetrahedra}\label{sec:tetraformulas}
An \emph{isosceles tetrahedron (disphenoid)} is a tetrahedron in which each pair of opposite edges has equal length. Consequently, all faces are congruent acute triangles. Leech~\cite{MR38667} summarizes this and other noteworthy properties of disphenoids.

The properties required in our context can be derived algebraically by inscribing a disphenoid with edge lengths $a$, $b$, and $c$ into a cuboid whose vertices have Cartesian coordinates $(\pm x_0,\pm y_0,\pm z_0)$, where
\begin{equation}\label{equ:xyz}
	x_0=\frac{\sqrt{-a^2+b^2+c^2}}{2\sqrt{2}}, \qquad
	y_0=\frac{\sqrt{a^2-b^2+c^2}}{2\sqrt{2}}, \qquad
	z_0=\frac{\sqrt{a^2+b^2-c^2}}{2\sqrt{2}}.
\end{equation}
Since the face triangles are acute, the quantities $x_0$, $y_0$, and $z_0$ are well defined.

\begin{figure}
	\centering
	\includegraphics[width=0.5\textwidth]{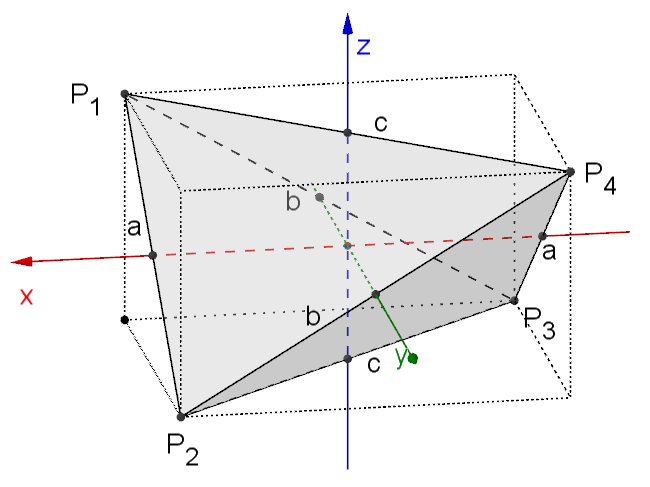}
	\caption{The edges of the isosceles tetrahedron appear as face diagonals of the cuboid.}
	\label{fig:disphenoid}
\end{figure}

For each pair of opposite edges of lengths $a$, $b$, and $c$, the line segment connecting their midpoints is perpendicular to the edges and has length $d_a=2x_0$, $d_b=2y_0$, and $d_c=2z_0$, respectively.
Explicitly,
\begin{equation}\label{equ:hinge_distances}
	d_a=\sqrt{\frac{-a^2+b^2+c^2}{2}}, \quad
	d_b=\sqrt{\frac{a^2-b^2+c^2}{2}}, \quad
	d_c=\sqrt{\frac{a^2+b^2-c^2}{2}}.
\end{equation}
When analyzing the kinematics of the kaleidocycle, these line segments correspond to the links of the threefold-symmetric Bricard linkage discussed in the next section.

With the tetrahedron vertices given by
\begin{equation}\label{equ:P1P2P3P4}
	P_1(x_0,-y_0,z_0), \quad
	P_2(x_0,y_0,-z_0), \quad
	P_3(-x_0,-y_0,-z_0), \quad
	P_4(-x_0,y_0,z_0),
\end{equation}
we can determine the skew angles $\alpha_a$, $\alpha_b$, and $\alpha_c$ between opposite edges.
Using equations~\eqref{equ:xyz} and~\eqref{equ:P1P2P3P4}, we obtain, for example,
\begin{equation}\label{equ:skew_angle_a}
	\cos\alpha_a
	= \frac{\lvert \overrightarrow{P_1P_2} \cdot \overrightarrow{P_4P_3} \rvert}
	{\lVert \overrightarrow{P_1P_2} \rVert\,\lVert \overrightarrow{P_4P_3} \rVert}
	= \frac{\lvert b^2-c^2\rvert}{a^2}.
\end{equation}

The skew angles are therefore given by
\begin{equation}\label{equ:skew_angles}
	\alpha_a=\arccos\frac{\lvert b^2-c^2\rvert}{a^2},
	\qquad
	\alpha_b=\arccos\frac{\lvert c^2-a^2\rvert}{b^2},
	\qquad 
	\alpha_c=\arccos\frac{\lvert a^2-b^2\rvert}{c^2}.
\end{equation}
We have $ \alpha_a, \alpha_b, \alpha_c \in (0,\tfrac{\pi}{2}]$.
 
We also require the relationship between the edge lengths and the dihedral
angles of an isosceles tetrahedron. A dihedral angle is the interior angle
between two adjacent faces. Leech~\cite{MR38667} gives a compact formula for
a dihedral angle which, besides the edge lengths, involves the area of the
congruent faces of the tetrahedron. For our purposes, however, we need an
expression for the dihedral angle purely in terms of $a$, $b$, and $c$.

\begin{theorem}\label{theo:dihedral1}
	The dihedral angle $\Phi_a$ of a disphenoid with edge lengths $a$, $b$, and $c$ along the edge of length $a$ satisfies
	\begin{equation}\label{equ:dihedral}
		\cos\Phi_a
		= \frac{a}{a+b-c}
		+ \frac{a}{a-b+c}
		+ \frac{a}{a+b+c}
		+ \frac{a}{a-b-c}
		- 1.
	\end{equation}
	By cyclic permutation of $a$, $b$, and $c$, one obtains the corresponding
	expressions for $\Phi_b$ and $\Phi_c$.
\end{theorem}

\begin{proof}
	Although Equation~\eqref{equ:dihedral} could be obtained from Leech’s
	formula~\cite{MR38667} combined with Heron’s formula, we choose a more
	direct derivation, which also allows us to derive a simple alternative
	expression for 	$\cos\Phi_a$ in terms of the parameters
	$a$, $d_a$, and $\alpha_a$ instead of $a$, $b$, and $c$.
	
	Let $\overrightarrow{n_1}$ and $\overrightarrow{n_2}$ be outward-pointing normal
	vectors of the faces $\Delta \mathrm{P_1P_2P_3}$ and
	$\Delta \mathrm{P_1P_2P_4}$ (see Figure~\ref{fig:disphenoid}). The
	dihedral angle $\Phi_a$ at the edge $\overline{\mathrm{P_1P_2}}$ of
	length $a$ is given by
	\begin{equation*}
		\cos \Phi_a =
		-\frac{\overrightarrow{n_1}\cdot\overrightarrow{n_2}}
		{\|\overrightarrow{n_1}\|\,\|\overrightarrow{n_2}\|}.
	\end{equation*}
	
	Using the coordinates in Equation~\eqref{equ:P1P2P3P4} and simplifying
	the resulting scalar products gives
	\begin{equation}\label{equ:cosphi_xyz}
		\cos\Phi_a
		=
		\frac{x_0^2(y_0^2+z_0^2)-y_0^2z_0^2}
		{x_0^2(y_0^2+z_0^2)+y_0^2z_0^2}.
	\end{equation}
	We shall also use this intermediate expression below to derive an
	alternative formula. Substituting the expressions for $x_0^2$, $y_0^2$, and $z_0^2$ from
	Equation~\eqref{equ:xyz} into Equation~\eqref{equ:cosphi_xyz} and
	simplifying yields
	\begin{equation*}
		\cos\Phi_a
		=
		-1+a\left(
		\frac{1}{a-b-c}
		+\frac{1}{a+b-c}
		+\frac{1}{a-b+c}
		+\frac{1}{a+b+c}
		\right),
	\end{equation*}
	
	which is equivalent to Equation~\eqref{equ:dihedral}. The corresponding symbolic calculations are documented in the supplementary Mathematica notebook.
\end{proof}

\begin{remark}
	The structure of Equation~\eqref{equ:dihedral} becomes more transparent when it is written in the form
	\begin{equation}
		\cos\Phi_a
		=
		\sum_{\pm,\pm}\frac{a}{a\pm b\pm c}-1 ,
	\end{equation}
	where the summation runs over the four independent choices of the signs in
	\(a \pm b \pm c\).
\end{remark}

\begin{theorem}\label{theo:dihedral2}
	The dihedral angle $\Phi_a$ of a disphenoid along the edge of length $a$,
	expressed in terms of the distance $d_a$ and the skew angle $\alpha_a$
	associated with the opposite edges of length $a$, satisfies
	\begin{equation}\label{equ:dihedral2}
		\cos\Phi_a=\frac{4d_a^2-a^2\sin^2\alpha_a}{4d_a^2+a^2\sin^2\alpha_a}.
	\end{equation}
\end{theorem}

\begin{proof}
	From Figure~\ref{fig:disphenoid} and Equation~\eqref{equ:xyz} we see that
	\begin{equation}\label{equ:phialt1}
		x_0=\frac{d_a}{2},\qquad y_0^2+z_0^2=\frac{a^2}{4}.
	\end{equation}

Projecting the right face of the cuboid onto the left face along the $x$-axis in
Figure~\ref{fig:disphenoid}, and drawing through the face center a line parallel
to the $y$-axis if $y_0\ge z_0$ or parallel to the $z$-axis if $y_0\le z_0$,
one sees that this line divides the skew angle $\alpha_a$ into two equal parts.
For $y_0\le z_0$ we have
\begin{equation*}
	y_0=\frac{a}{2}\sin\frac{\alpha_a}{2},\qquad
	z_0=\frac{a}{2}\cos\frac{\alpha_a}{2},
\end{equation*}
and vice versa for $y_0\ge z_0$. In both cases we obtain
\begin{equation}\label{equ:phialt2}
	y_0^2z_0^2
	=\frac{1}{16}a^4\sin^2\frac{\alpha_a}{2}\cos^2\frac{\alpha_a}{2}
	=\frac{1}{64}a^4\sin^2\alpha_a.
\end{equation}

	Substituting Equations~\eqref{equ:phialt1} and \eqref{equ:phialt2} into
	Equation~\eqref{equ:cosphi_xyz} immediately yields the result.
\end{proof}

In this context, we also note the following geometric property.
\begin{lemma}
	\label{lem:dihedralbisector}
	The bisector plane of a dihedral angle along an edge of an isosceles tetrahedron intersects the opposite edge at its midpoint.
\end{lemma}

\begin{proof}
	Let $AB$ and $CD$ be opposite edges of the disphenoid, and let $N$ and $M$ be their respective midpoints. As shown above using the cuboid representation, the segment joining the midpoints of two opposite edges is perpendicular to both edges.
	
	A rotation through $180^\circ$ about $MN$ therefore interchanges $A$ with $B$ and $C$ with $D$, and consequently maps the face $ABC$ onto the face $ABD$. Since the plane through $A$, $B$, and $M$ is invariant under this rotation, the two faces make equal angles with it. Hence this plane bisects the dihedral angle along $AB$.
\end{proof}

\section{The Kinematics of the Linkage}\label{sec:kinematicslinkage}

A kaleidocycle consisting of six isosceles tetrahedra can be identified with a threefold-symmetric Bricard 6R linkage (see Figure~\ref{fig:linkagebig}), a class of mechanisms studied in 2005 by Chen et al.~\cite{chen2005threefold} in a different context.

\begin{figure}
	\centering
	\includegraphics[width=0.5\textwidth]{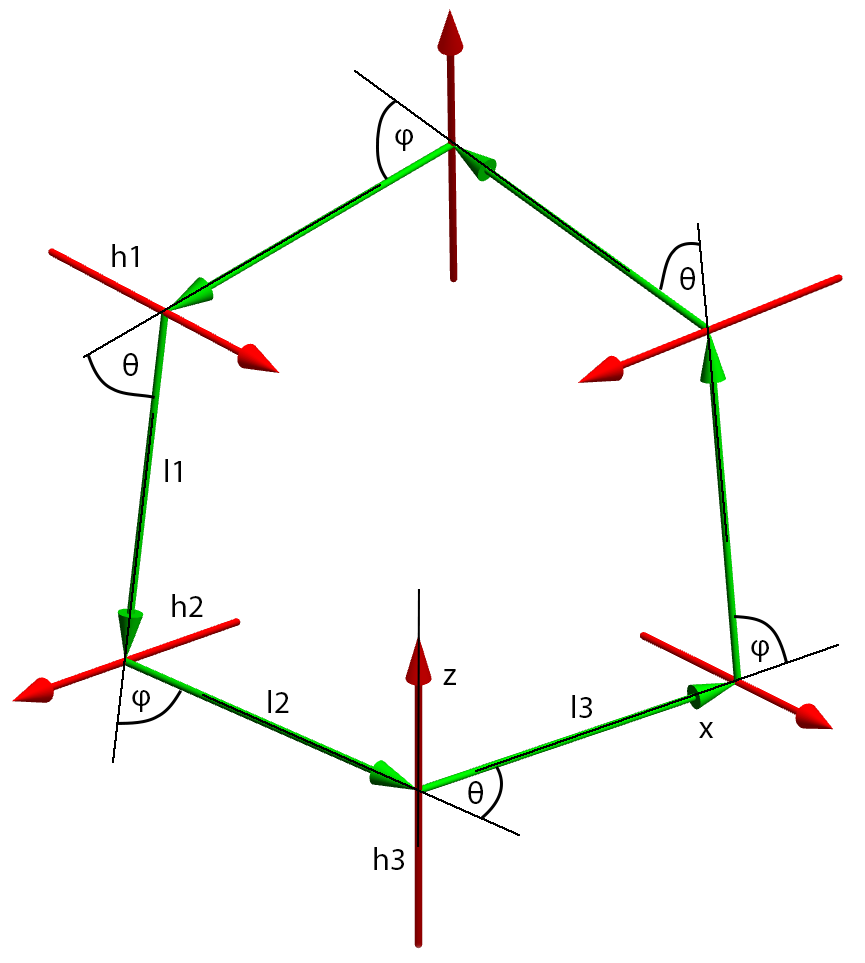}
	\caption{Threefold-symmetric Bricard 6R linkage.}
	\label{fig:linkagebig}
\end{figure}

The kinematics of the linkage can be described in terms of three angles $\alpha$, $\theta$, and $\varphi$. The angle $\alpha$ denotes the oriented twist angle between adjacent hinge axes, whose magnitude is constant and determined by the skew angle of the disphenoid defined in~\eqref{equ:skew_angles}. Along the cycle, the twists alternate between $\alpha$ and $-\alpha$ (equivalently, $2\pi-\alpha$). The variables $\theta$ and $\varphi$ represent the hinge rotation angles, which alternate sequentially at consecutive joints around the loop due to the threefold symmetry.

Each link connects the midpoints of the two opposite hinge edges of a disphenoid. At each hinge, the zero position of the hinge rotation angle is defined by the configuration in which the two adjacent links lie on the same straight line, extending in opposite directions from their common hinge. The angles $\theta$ and $\varphi$ measure the rotations of the following link relative to the preceding link from this reference position, with the positive direction determined by the oriented hinge axes shown in Figure~\ref{fig:linkagebig}, according to the right-hand rule.

Using the Denavit--Hartenberg formalism~\cite{Denavit1955}, the link length is $d_a$, the perpendicular distance between opposite hinge edges of a disphenoid, as established in Section~\ref{sec:tetraformulas}. Since the common perpendicular bisects both edges, all joint offsets vanish. Chen et al.~\cite{chen2005threefold} used this formalism to derive a closure equation for the threefold-symmetric Bricard 6R linkage. They also established that this linkage possesses three reflection planes containing its rotational axis throughout the motion, giving it point group $C_{3v}$. Here, we derive an equivalent but more compact form of the closure equation directly from its threefold symmetry.

To obtain this compact form, we describe the threefold symmetry in terms of rotations. We use the following standard result for the composition of two rotations, which follows naturally from their quaternion representation and is stated here in the form given by Kuipers~\cite[Ch.~6.3]{kuipers1999quaternions}.

\begin{lemma}\label{lemma:rotmult}
	Let $R_\theta$ be a rotation about an axis $\vec n_\theta$ by angle $\theta$, 
	and let $R_\varphi$ be a rotation about an axis $\vec n_\varphi$ by angle $\varphi$. 
	Assume both rotations are taken about axes through the origin. 
	Let $\alpha$ denote the angle between $\vec n_\theta$ and $\vec n_\varphi$. 
	
	Then the composition $R_\varphi\circ R_\theta$ can be represented as a rotation through an angle $\gamma\in[0,2\pi]$ satisfying
	
	\begin{equation}\label{equ:rotmult}
		\cos\frac{\gamma}{2}
		=
		\cos\frac{\theta}{2}\cos\frac{\varphi}{2}
		-
		\cos\alpha\,\sin\frac{\theta}{2}\sin\frac{\varphi}{2}.
	\end{equation}
\end{lemma}

Using Lemma~\ref{lemma:rotmult}, we can derive a closure equation directly, with almost no computation.
 \begin{theorem}\label{theorem:phi_theta_short}
	The closure equation for a threefold-symmetric Bricard \(6R\) linkage with twist angle $\alpha$ and hinge rotation angles $\theta$ and $\varphi$ is given by
	\begin{equation}\label{equ:closureshort}
		\cos\frac{\theta}{2}\,\cos\frac{\varphi}{2}
		-
		\cos\alpha\,\sin\frac{\theta}{2}\,\sin\frac{\varphi}{2}
		= \pm\frac{1}{2}.
	\end{equation}
\end{theorem}

\begin{proof}
	We place the linkage in space such that link $l_3$ lies on the $x$-axis and hinge $h_3$ lies on the $z$-axis (see Figure~\ref{fig:linkagebig}). Without loss of generality, we assume all links have unit length. We denote by $L_3$ the position of $l_3$ and by $H_3$ the position of $h_3$. Our goal is to map the configuration so that $l_1$ is mapped to $L_3$ and $h_1$ is mapped to $H_3$ by rigid motions.
	
	First, we apply a rotation $Z_\theta$ about the $z$-axis, aligning $l_2$ with the $x$-axis. A subsequent translation by one unit in the $x$-direction moves $l_2$ to $L_3$.Next, we apply a rotation $X_{2\pi-\alpha}$ about the $x$-axis, which fixes $l_2$ at $L_3$ and maps $h_2$ to $H_3$.\footnote{Since the twist along the cycle alternates between $\alpha$ and $2\pi-\alpha$, the initial position must be chosen accordingly.}
	
	In an analogous way, we map $l_1$ and $h_1$ to $L_3$ and $H_3$, respectively. We first apply a rotation $Z_\varphi$ about the $z$-axis, aligning $l_1$ with the $x$-axis, followed by a translation by one unit in the $x$-direction. Finally, we apply a rotation $X_\alpha$ about the $x$-axis, which fixes $l_1$ at $L_3$ and maps $h_1$ to $H_3$.

	The overall transformation is a rigid motion of the linkage. By the threefold symmetry of the structure, this rigid motion corresponds to a net global rotation $\gamma$ about the symmetry axis.
	
	Since we consider an already closed threefold-symmetric Bricard 6R linkage, spatial closure is part of the geometric setting. To derive the corresponding relation between the hinge angles, we need only consider the rotational part of the rigid motion. Evaluating from right to left, the rotational component of the combined transformation is
	\[
	R=X_\alpha\circ Z_\varphi\circ X_{2\pi-\alpha}\circ Z_\theta.
	\]

	Since $X_\alpha\circ(\cdot)\circ X_{2\pi-\alpha}$ is a conjugation by a rotation about the $x$-axis, it preserves rotation angles and sends a rotation to one whose axis is obtained by rotating the original axis about the $x$-axis by the angle $\alpha$. Therefore, the sub-composition
	\[
	R_\varphi=X_\alpha\circ Z_\varphi\circ X_{2\pi-\alpha}
	\]
	is a rotation by angle $\varphi$ whose axis is separated from the $z$-axis by exactly $\alpha$.
	
	The entire rotational transformation can now be written as the composition of two rotations,
	\[
	R=R_\varphi\circ Z_\theta.
	\]
	Applying Lemma~\ref{lemma:rotmult} directly to this composition yields
	\[
	\cos\frac{\gamma}{2}
	=\cos\frac{\theta}{2}\cos\frac{\varphi}{2}
	-\cos\alpha\sin\frac{\theta}{2}\sin\frac{\varphi}{2}.
	\]
	
	By the threefold symmetry of the linkage, $R$ is a rotation through $120^\circ$ about the symmetry axis. The same rotation can also be described as a rotation through $240^\circ$ about the oppositely directed axis. Thus, the angle $\gamma$ in Lemma~\ref{lemma:rotmult} can take either value, giving
	\[
	\cos\frac{\gamma}{2}
	=\cos60^\circ=\frac12
	\qquad\text{or}\qquad
	\cos\frac{\gamma}{2}
	=\cos120^\circ=-\frac12.
	\]
	
	Consequently,
	\[
	\cos\frac{\theta}{2}\cos\frac{\varphi}{2}
	-\cos\alpha\sin\frac{\theta}{2}\sin\frac{\varphi}{2}
	=\pm\frac12,
	\]
	which completes the proof.
\end{proof}

We draw some conclusions which hold for arbitrary twist angles $\alpha$:

\begin{remark}\label{rem:plus-half}
	In the next section, we consider only the $+\frac12$ case in Equation~\eqref{equ:closureshort}. Shifting either hinge angle by $2\pi$ maps the $+\frac12$ branch directly onto the $-\frac12$ branch, meaning that both branches yield identical physical linkage configurations over a full revolution.
\end{remark}

\begin{remark}\label{rem:symmetric}
	Equation~\eqref{equ:closureshort} is symmetric in the angles $\theta$ and $\varphi$.
\end{remark}

\begin{remark}\label{rem:negated-solution}
	If $(\theta,\varphi)$ solves Equation~\eqref{equ:closureshort}, then so does $(-\theta,-\varphi)$.
\end{remark}

\begin{remark}\label{rem:twist-symmetry}
	The relative positions of the links during the motion are the same for twist
	angles $\alpha$ and $\pi-\alpha$; the two linkages therefore exhibit identical
	kinematic behavior.
\end{remark}

This is because the directions of the hinge axes in Figure~\ref{fig:linkagebig}
are arbitrary. For example, if the directions of all three axes
associated with the angle $\theta$ are reversed, the twist angle of the linkage
alternates between $\pi+\alpha$ and $\pi-\alpha$ instead of $\alpha$ and
$-\alpha$, while $\theta$ is replaced by $2\pi-\theta$ (or equivalently
$-\theta$). This transformation does not affect the kinematics of the linkage.
More formally, Equation~\eqref{equ:closureshort} is invariant under the
substitution $\alpha \mapsto \pi-\alpha$, $\theta \mapsto -\theta$.
Figures~\ref{fig:alphagraphs}(c) and~\ref{fig:alphagraphs}(e) show an example.

Next, we demonstrate that our compact closure equation is mathematically equivalent to the classical kinematic formulation established by Chen et al.~\cite{chen2005threefold}.

\begin{theorem}\label{theorem:chen_equivalence}
	The compact half-angle closure equation \eqref{equ:closureshort} is mathematically equivalent to Chen's classical full-angle formulation.
\end{theorem}

\begin{proof}
	We can verify this equivalence by expanding our compact form to recover Chen's original equation. We start with the half-angle expression:
	\begin{equation*}
		\cos\frac{\theta}{2}\,\cos\frac{\varphi}{2}
		-
		\cos\alpha\,\sin\frac{\theta}{2}\,\sin\frac{\varphi}{2}
		= \pm\frac{1}{2}.
	\end{equation*}
	
	Squaring both sides eliminates the sign ambiguity and yields:
	\begin{equation*}
		\cos^2\frac{\theta}{2}\cos^2\frac{\varphi}{2} 
		+ \cos^2\alpha\sin^2\frac{\theta}{2}\sin^2\frac{\varphi}{2} 
		- 2\cos\alpha\cos\frac{\theta}{2}\cos\frac{\varphi}{2}\sin\frac{\theta}{2}\sin\frac{\varphi}{2} = \frac{1}{4}.
	\end{equation*}
	
	Multiplying the entire equation by $4$ allows us to directly apply standard double-angle relations ($2\sin\tfrac{x}{2}\cos\tfrac{x}{2} = \sin x$, $2\cos^2\tfrac{x}{2} = 1+\cos x$, and $2\sin^2\tfrac{x}{2} = 1-\cos x$):
	\begin{equation}\label{equ:double_angle_applied}
		(1+\cos\theta)(1+\cos\varphi) + \cos^2\alpha(1-\cos\theta)(1-\cos\varphi) - 2\cos\alpha\sin\theta\sin\varphi = 1.
	\end{equation}
	
	Multiplying out the products in \eqref{equ:double_angle_applied} completely yields:
	\begin{equation*}
		(1 + \cos\theta + \cos\varphi + \cos\theta\cos\varphi) + (\cos^2\alpha - \cos^2\alpha\cos\theta - \cos^2\alpha\cos\varphi + \cos^2\alpha\cos\theta\cos\varphi) - 2\cos\alpha\sin\theta\sin\varphi = 1.
	\end{equation*}
	
	Subtracting $1$ from both sides eliminates the constant $1$ on the far left and the far right, leaving $\cos^2\alpha$ as the sole constant term. Grouping the remaining terms by their trigonometric factors gives:
	\begin{equation*}
		\cos^2\alpha + (1 - \cos^2\alpha)\cos\theta + (1 - \cos^2\alpha)\cos\varphi + (1 + \cos^2\alpha)\cos\theta\cos\varphi - 2\cos\alpha\sin\theta\sin\varphi = 0.
	\end{equation*}
	
    Factoring out $(\cos\theta + \cos\varphi)$ from the linear terms and substituting the identity $1 - \cos^2\alpha = \sin^2\alpha$, we arrive precisely at Chen's original full-angle closure equation:
\[
\cos^2\alpha
+ \sin^2\alpha\,(\cos\theta + \cos\varphi)
+ (1+\cos^2\alpha)\,\cos\theta\,\cos\varphi
- 2\cos\alpha\,\sin\theta\,\sin\varphi = 0.
\]
Conversely, since this expansion consists entirely of reversible trigonometric substitutions, the steps can be followed backward to factor Chen's equation into a perfect square, recovering exactly our $\pm 1/2$ branches. This establishes full logical equivalence and completes the proof.
\end{proof}

This shows that both equations describe the exact same kinematics. However, our new equation separates the joint variables into simple half-angle products, making it far more straightforward to isolate individual angles analytically.

\section{The Kinematics of the Kaleidocycle}\label{sec:kinematicskaleido}

Using isosceles tetrahedra as links introduces additional geometric constraints when a continuous rotational motion of the kaleidocycle is required. As explained in Remark~\ref{rem:twist-symmetry}, the kinematics is unchanged when the twist angle $\alpha$ is replaced by $\pi-\alpha$. We may therefore choose $\alpha$ to coincide with the skew angle $\alpha_a$ of the tetrahedra defined in~\eqref{equ:skew_angles}. The hinge rotation angles $\theta$ and $\varphi$ are naturally restricted to the interval $(-\pi,\pi)$.

For a given twist angle $\alpha$, Equation~\eqref{equ:closureshort} allows us to compute the admissible values of $\varphi$ for each value of $\theta$.

\begin{theorem}\label{theorem:phi_theta}
	For a kaleidocycle with twist angle $\alpha$, the hinge angle $\varphi$ is explicitly expressed
	in terms of the hinge angle $\theta$ by
	\begin{equation}\label{equ:phi}
		\varphi_\alpha(\theta)
		=
		4\arctan\!\left(
		\frac{\pm\sqrt{3-4\sin^2\!\alpha\,\sin^2\!\tfrac{\theta}{2}}
			-2\cos\alpha\,\sin\tfrac{\theta}{2}}
		{2\cos\tfrac{\theta}{2}+1}
		\right).
	\end{equation}
\end{theorem}
\begin{proof}
	We use the well-known
	tangent half-angle substitution to derive an explicit expression for
	$\varphi_\alpha(\theta)$. With
	\begin{equation}\label{equ:t-subst}
		t=\tan\!\left(\frac{\varphi}{4}\right), \qquad -\pi<\varphi<\pi,
	\end{equation}
	we have $-1<t<1$, and
	\begin{equation}\label{equ:halfangle}
		\cos\!\left(\frac{\varphi}{2}\right)=\frac{1-t^2}{1+t^2}, \qquad
		\sin\!\left(\frac{\varphi}{2}\right)=\frac{2t}{1+t^2}.
	\end{equation}
	
	Substituting~\eqref{equ:halfangle} into
	equation~\eqref{equ:closureshort} with $+\tfrac{1}{2}$ on the right-hand
	side (see Remark~\ref{rem:plus-half}), eliminating the common denominator
	$1+t^2$, and simplifying, we obtain
	\[
	t^{2}\bigl(1+2\cos\tfrac{\theta}{2}\bigr)
	+4t\cos\alpha\sin\tfrac{\theta}{2}
	-2\cos\tfrac{\theta}{2}+1
	=0.
	\]
	Solving this quadratic equation for $t$ yields
	\begin{equation}\label{equ:t-solve}
		t=\frac{\pm\sqrt{3-4\sin^2\!\alpha\,\sin^2\!\tfrac{\theta}{2}}
			-2\cos\alpha\,\sin\tfrac{\theta}{2}}
		{2\cos\tfrac{\theta}{2}+1}.
	\end{equation}
	Re-substituting~\eqref{equ:t-subst} into~\eqref{equ:t-solve} finally gives Equation~\eqref{equ:phi}.
\end{proof}
Safsten et al.~\cite{safsten2016} present an alternative formulation of $\varphi_\alpha(\theta)$ whose analytical structure is more involved\footnotemark.

\footnotetext{\label{fn:altphi}Safsten et al.~\cite{safsten2016} give the following alternative explicit expression for the hinge angle $\varphi$:
	\[
	\begin{aligned}
		\varphi_\alpha(\theta)
		={}&
		\pm \arccos\!
		\left(
		-\frac{\cos^2\alpha\,(1-\cos\theta)+\cos\theta}
		{\cos^2\alpha\,(1-\cos\theta)+(1+\cos\theta)}
		\right) \\
		&\quad
		+ \arg\!\left[
		\left(-\cos^2\alpha\,(1-\cos\theta)+(1+\cos\theta)\right)
		- i\left(2\cos\alpha\,\sin\theta\right)
		\right].
	\end{aligned}
	\]
}

We decompose $\varphi_\alpha(\theta)$ into the functions
$\varphi_\alpha^{+}(\theta)$ and $\varphi_\alpha^{-}(\theta)$ and plot
the corresponding graphs for different values of $\alpha$ (see Figure~\ref{fig:alphagraphs}).

\begin{figure}
	\centering
	\setlength{\tabcolsep}{2pt}
	\begin{tabular}{@{}ccccc@{}}
		\includegraphics[width=0.19\textwidth]{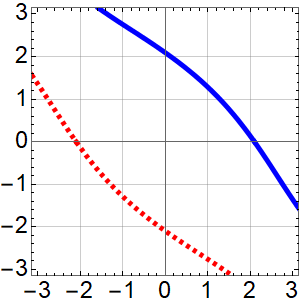} &
		\includegraphics[width=0.19\textwidth]{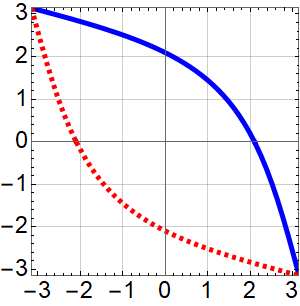} &
		\includegraphics[width=0.19\textwidth]{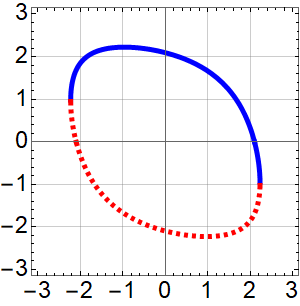} &
		\includegraphics[width=0.19\textwidth]{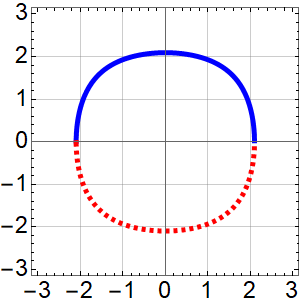} &
		\includegraphics[width=0.19\textwidth]{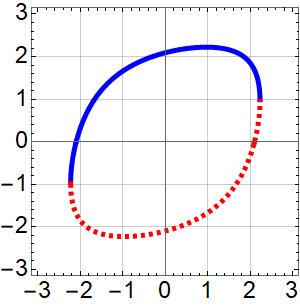} \\
		
		(a) $\alpha=\pi/4$ &
		(b) $\alpha=\pi/3$ &
		(c) $\alpha=5\pi/12$ &
		(d) $\alpha=\pi/2$ &
		(e) $\alpha=7\pi/12$
	\end{tabular}

\caption{Relation between the hinge rotation angles $\theta$ and $\varphi$,
	shown on the horizontal and vertical axes, respectively, in
	$[-\pi,\pi]\times[-\pi,\pi]$. Solid line: $\varphi_\alpha^+(\theta)$;
	dashed line: $\varphi_\alpha^-(\theta)$.
	(a) $\alpha=\pi/4$: no closed path.
	(b) $\alpha=\pi/3$: degenerate closed path.
	(c)--(e) Closed paths for $\alpha=5\pi/12$, $\pi/2$, and $7\pi/12$,
	respectively. Figures~(c) and~(e) illustrate Remark~\ref{rem:twist-symmetry}.}

	\label{fig:alphagraphs}
\end{figure}

A continuous periodic rotational motion is only possible if 
$\varphi^+_\alpha(\theta)$ and $\varphi^-_\alpha(\theta)$ together form a closed loop. 
In the context of linkage kinematics, this one-parameter curve is called a \emph{compatibility path}, 
and moving along this path corresponds to the continuous motion.

\begin{lemma}\label{lemma:closedpath}
	The compatibility path of a kaleidocycle is closed if and only if
	the twist angle $\alpha$ satisfies
	\begin{equation}\label{equ:alpharange}
		\frac{\pi}{3}<\alpha<\frac{2\pi}{3}.
	\end{equation}
\end{lemma}
\begin{proof}
	For admissible values of $\theta$, the two solution branches
	$\varphi^-_\alpha(\theta)$ and $\varphi^+_\alpha(\theta)$
	of equation~\eqref{equ:phi} meet precisely where the radicand vanishes.
	Thus, a closed compatibility path exists exactly when there is
	some $\theta_0\in(0,\pi)$ such that
	\[
	3-4\sin^2\alpha\sin^2\frac{\theta_0}{2}=0.
	\]
	Equivalently,
	\[
	\sin^2\frac{\theta_0}{2}
	=\frac{3}{4\sin^2\alpha}.
	\]
	A solution $\theta_0\in(0,\pi)$ exists if and only if
	$\sin^2\alpha>3/4$, which for $0<\alpha<\pi$
	is equivalent to inequality~\eqref{equ:alpharange}.
\end{proof}

\begin{remark}
	Figure~\ref{fig:alphagraphs}(b) also shows a closed path for $\alpha=\pi/3$. However, this degenerate case, as well as the corresponding case $\alpha=2\pi/3$, leads to hinge rotation angles ranging over the full interval $[-\pi,\pi]$, which cannot occur in a kaleidocycle, since the dihedral angle of a tetrahedron at the hinge edge is strictly positive.
\end{remark}

From equation~\eqref{equ:phi}, we also obtain the following
bounds on the hinge angles, which are essential for determining
the realizability of a kaleidocycle.

\begin{theorem}\label{theorem:anglerange}
	For a kaleidocycle with twist angle $\alpha$ in the range given in
	Lemma~\ref{lemma:closedpath}, the hinge angles $\theta$ and $\varphi$
	satisfy
\begin{equation}\label{equ:thetainterval}
	-2\arcsin\left(\frac{\sqrt{3}}{2\sin\alpha}\right)
	\leq \theta,\varphi \leq
	2\arcsin\left(\frac{\sqrt{3}}{2\sin\alpha}\right).
\end{equation}
\end{theorem}
\begin{proof}
	For $\alpha$ in the range of Lemma~\ref{lemma:closedpath},
	a vertical line $\theta=\mathrm{const}$ intersects the compatibility
	path in zero, one, or two points, depending on whether the radicand
	in equation~\eqref{equ:phi} is negative, vanishes, or positive,
	respectively.
	As shown in the proof of Lemma~\ref{lemma:closedpath},
	the radicand vanishes precisely when
	\[
	\sin^2\frac{\theta_0}{2}
	=\frac{3}{4\sin^2\alpha}.
	\]
	The vertical lines $\theta=-\theta_0$ and $\theta=\theta_0$ are tangent to the compatibility path and mark the minimum and maximum hinge angles attained during the cyclic motion, respectively.
	
	Taking $\theta_0>0$, we obtain
	\[
	\theta_0=
	2\arcsin\left(\frac{\sqrt{3}}{2\sin\alpha}\right).
	\]
	Because of the symmetry between the rotation angles
	$\theta$ and $\varphi$, the same bounds hold for $\varphi$,
	which proves inequality~\eqref{equ:thetainterval}.
\end{proof}

Safsten et al.~\cite{safsten2016} obtain equivalent bounds using their alternative formulation of Equation~\eqref{equ:phi} (see the corresponding footnote on page~\pageref{fn:altphi}).

\section{Realizable Kaleidocycles}\label{sec:realizable}

We now investigate under which geometric conditions a threefold-symmetric kaleidocycle can complete a full cycle of motion without blocking. Our analysis assumes perfectly rigid tetrahedra whose interiors remain pairwise disjoint throughout the motion. Physical models may nevertheless remain movable even when this condition is violated, owing to manufacturing tolerances and small deformations. An example is discussed in Section~\ref{sec:doecahedron}.

We begin by introducing the following definitions.

A \emph{configuration} of a kaleidocycle specifies the positions and orientations of all tetrahedra such that the hinge constraints are satisfied; it represents a particular state of the kaleidocycle during its motion.

A configuration of a kaleidocycle is called \emph{blocking} in a given direction of motion if any further motion in that direction, however small, would cause the interiors of two tetrahedra to intersect.

As noted in the Introduction, a kaleidocycle is called realizable if its constituent tetrahedra can be assembled into a closed ring and undergo a complete cycle of motion without their interiors intersecting (see Figure~\ref{fig:configurations}).

Two tetrahedra are said to be in
\emph{face contact} if they share a common triangular face without gaps or overlap.

A configuration is called \emph{extremal} if either of the hinge angles $\theta$ or $\varphi$ attains its minimal or maximal value as specified in Theorem~\ref{theorem:anglerange}.

\begin{figure}
	\centering
	\begin{tabular}{ccc}
		\includegraphics[width=0.30\textwidth]{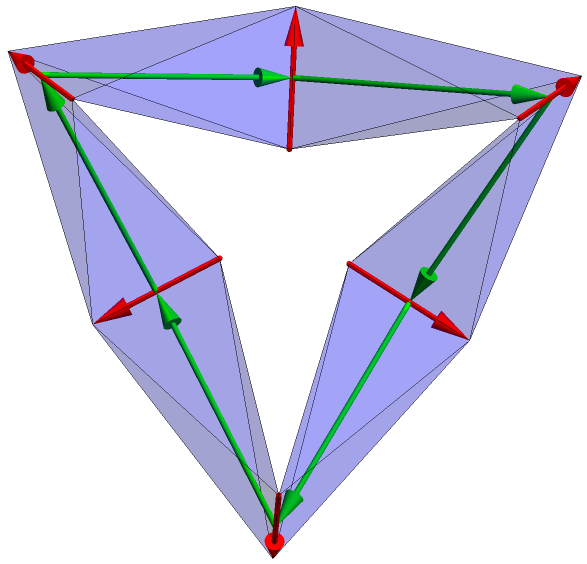} &
		\includegraphics[width=0.30\textwidth]{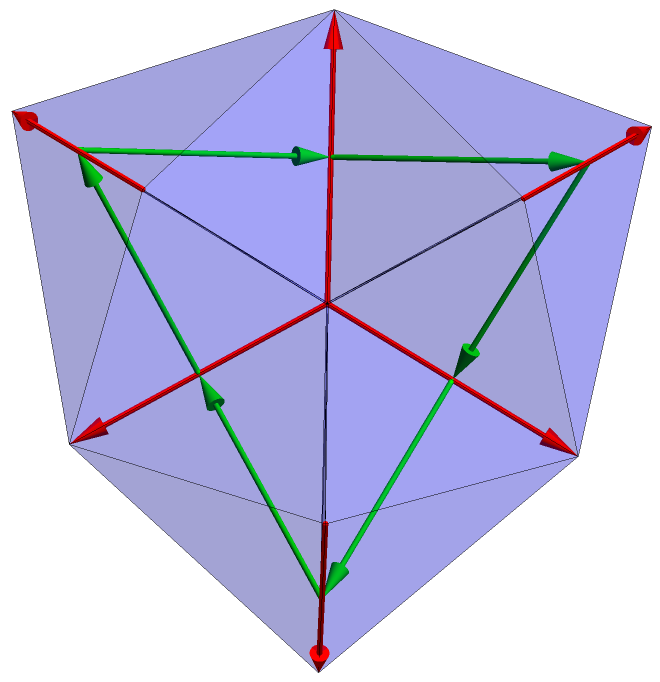} &
		\includegraphics[width=0.30\textwidth]{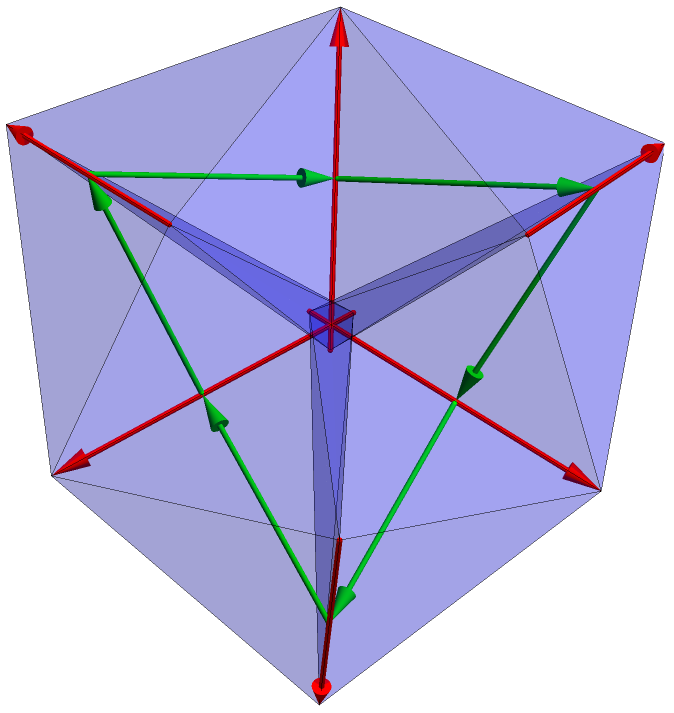} \\[2pt]
		(a) realizable & (b) realizable & (c) not realizable
	\end{tabular}
	
	\caption{Kaleidocycles at the extremal hinge angles.
		(a), (b) Realizable cases.
		(b) Adjacent tetrahedra are in face contact but do not intersect.
		(c) A non-realizable case: the corresponding position is unreachable
		because the tetrahedra would intersect.}
	\label{fig:configurations}
\end{figure}

Inequality~\eqref{equ:alpharange} provides only a necessary condition for
the realizability of a kaleidocycle composed of isosceles tetrahedra,
since the dihedral angles between adjacent tetrahedra at the hinges
impose additional restrictions on the admissible configurations.
Even if Inequality~\eqref{equ:alpharange} is satisfied, the motion may
still become blocked if the available angular range does not allow the
hinge rotation angles $\theta$ and $\varphi$ to attain the extremal
values prescribed by Inequality~\eqref{equ:thetainterval}.
In the following, we derive a necessary and sufficient condition in terms
of the geometric dimensions of the tetrahedra. This derivation reduces
the blocking analysis to pairs of adjacent tetrahedra. To justify this
reduction, we first show that non-adjacent tetrahedra cannot cause
blocking during the motion.

\begin{lemma}\label{lemma:nonadjacent}
	During the motion of a threefold-symmetric kaleidocycle, blocking can occur only through contact between adjacent tetrahedra.
\end{lemma}

\begin{proof}
	As established in Section~\ref{sec:kinematicslinkage}, the kaleidocycle possesses three reflection planes containing its rotational axis throughout the motion. The idea of the proof is that, as long as the central opening (the hole enclosed by the ring of tetrahedra) remains, the reflection symmetries prevent non-adjacent tetrahedra from intersecting.

	The three reflection planes divide space into six sectors of angle $60^\circ$. In the ring configuration, the six hinge edges lie in the corresponding symmetry half-planes, arranged in cyclic order around the central opening. Each tetrahedron has two opposite hinge edges containing all four of its vertices. Since these edges lie in neighboring half-planes and the tetrahedron is convex, it is entirely contained in the sector between them. Thus, the six tetrahedra occupy separate sectors, and non-adjacent tetrahedra cannot intersect.

	This arrangement can change only if a hinge edge reaches the rotational axis, since the axis separates the two half-planes of its reflection plane. Its first contact with the axis occurs at an endpoint. By threefold symmetry, the corresponding endpoints of three alternating hinge edges then coincide on the axis, obstructing the central opening. Two tetrahedra joined by one of the remaining hinges already share the two endpoints of their common hinge edge. At this position, one additional vertex of each tetrahedron coincides at the same point on the rotational axis. Since these three vertices form a triangular face of each tetrahedron, the two faces coincide. The tetrahedra are therefore in face contact, and further motion in this direction would cause their interiors to intersect.

	Hence each tetrahedron remains confined to its original $60^\circ$ sector until adjacent tetrahedra come into face contact. Consequently, blocking can occur only between adjacent tetrahedra.
\end{proof}

\begin{theorem}\label{theorem:nonblocking}
	A kaleidocycle composed of three isosceles tetrahedra
	and their three mirror images is realizable with the
	hinges located along edges of length $a$ if and only if
	\begin{equation}\label{equ:nonblock_simple}
		\sin\alpha\cos\frac{\Phi_a}{2}
		\geq \frac{\sqrt{3}}{2},
	\end{equation}
	where $\alpha$ is the skew angle and $\Phi_a$ is
	the internal dihedral angle at the hinge edge.
\end{theorem}

\begin{proof}

	First, assume that the kaleidocycle is realizable. By Lemma~\ref{lemma:nonadjacent}, blocking can occur only between adjacent tetrahedra.

	By Lemma~\ref{lem:dihedralbisector}, the link of each tetrahedron lies in the bisector plane of its hinge dihedral angle. At the reference hinge angle $\theta=0$, the links of two adjacent tetrahedra lie on the same straight line, extending in opposite directions from their common hinge. The angle between these links is therefore $\pi$.
	
	When one tetrahedron rotates relative to the other through an angle $\theta\in[-\pi,\pi]$, the angle between their links becomes $\pi-|\theta|$. Since each link bisects the dihedral angle $\Phi_a$ of its tetrahedron, the interiors of the two tetrahedra remain disjoint precisely when the angle between their links is at least $\Phi_a$. Thus,
	\[
	\pi-|\theta|\geq\Phi_a,
	\]
	or equivalently,
	\[
	|\theta|\leq\pi-\Phi_a.
	\]
	The same argument applies to $\varphi$. Hence the available angular range is symmetric about zero and has total width $2\pi-2\Phi_a$.

	According to Theorem~\ref{theorem:anglerange}, the angular range required for the motion is
	\[
	4\arcsin\left(\frac{\sqrt{3}}{2\sin\alpha}\right).
	\]
	Since both angular ranges are symmetric about zero, the motion remains unblocked if and only if
	\begin{equation}\label{equ:nonblock}
		4\arcsin\left(\frac{\sqrt{3}}{2\sin\alpha}\right)
		\leq 2\pi-2\Phi_a.
	\end{equation}
	
	Dividing by four gives
	\[
	\arcsin\left(\frac{\sqrt{3}}{2\sin\alpha}\right)
	\leq\frac{\pi-\Phi_a}{2}.
	\]
	Since both sides lie in $[0,\pi/2]$, taking
	the sine yields the equivalent inequality
	\[
	\frac{\sqrt{3}}{2\sin\alpha}
	\leq\cos\frac{\Phi_a}{2},
	\]
	which is precisely
	Inequality~\eqref{equ:nonblock_simple}.
	
	Conversely, assume that
	Inequality~\eqref{equ:nonblock_simple} holds.
	Since
	\[
	0<\cos\frac{\Phi_a}{2}<1,
	\]
	we obtain
	\[
	\sin\alpha>\frac{\sqrt{3}}{2},
	\]
	and hence
	\begin{equation}\label{equ:alpha_range}
		\frac{\pi}{3}<\alpha<\frac{2\pi}{3}.
	\end{equation}
	By Lemma~\ref{lemma:closedpath}, the compatibility
	path is therefore closed.
	
	Moreover, Inequality~\eqref{equ:nonblock_simple}
	is equivalent to Inequality~\eqref{equ:nonblock},
	which guarantees that adjacent tetrahedra do not
	block the motion. By Lemma~\ref{lemma:nonadjacent},
	non-adjacent tetrahedra cannot cause blocking either.
	Hence the kaleidocycle is realizable.
\end{proof}

\begin{lemma}\label{lemma:uv}
	For a disphenoid with edge lengths $a$, $b$, and $c$, set
	\[
	u=\frac{|b^2-c^2|}{a^2},\qquad v=\frac{b^2+c^2}{a^2}.
	\]
	The corresponding kaleidocycle with hinges along the edges of length $a$ is realizable if and only if
	\begin{equation}\label{equ:uv_bound}
		0\leq u<\frac12,\qquad v\geq\frac{5-11u^2}{2-8u^2}.
	\end{equation}
\end{lemma}

\begin{proof}
	If the kaleidocycle is realizable, Theorem~\ref{theorem:nonblocking} implies $\sin\alpha>\sqrt{3}/2$. Since $u=\cos\alpha$ by Equation~\eqref{equ:skew_angles}, we obtain $0\leq u<1/2$.
	
	Using the formula for the skew angle $\alpha$ in Equation~\eqref{equ:skew_angles} and the formula for $d_a$ in Equation~\eqref{equ:hinge_distances}, we obtain
	\[
	\sin^2\alpha=1-u^2,\qquad d_a^2=\frac{-a^2+b^2+c^2}{2}=\frac{a^2}{2}(v-1).
	\]
	Combining the half-angle identity with the dihedral-angle formula of Theorem~2 yields
	\[
	\cos^2\frac{\Phi_a}{2}
	=\frac{1+\cos\Phi_a}{2}
	=\frac{4d_a^2}{4d_a^2+a^2\sin^2\alpha}
	=\frac{2(v-1)}{2v-1-u^2}.
	\]
	Since both factors on the left-hand side of the realizability condition~\eqref{equ:nonblock_simple} are positive, squaring and substituting the above expressions yields the equivalent inequality
	\[
	\frac{2(1-u^2)(v-1)}{2v-1-u^2}\geq\frac34.
	\]
	Since $v>1$ and $u<1/2$, the denominator is strictly positive. Clearing denominators gives
	\[
	8(1-u^2)(v-1)\geq 3(2v-1-u^2).
	\]
	Expanding and rearranging yields
	\[
	(2-8u^2)v\geq 5-11u^2.
	\]
	Since $0\leq u<1/2$, the coefficient $2-8u^2$ is strictly positive, and hence
	\[
	v\geq\frac{5-11u^2}{2-8u^2}.
	\]
	
	Conversely, if both conditions in \eqref{equ:uv_bound} hold, reversing the algebraic transformations yields
	\[
	\sin^2\alpha\cos^2\frac{\Phi_a}{2}\geq\frac34.
	\]
	Since both $\sin\alpha$ and $\cos(\Phi_a/2)$ are positive, taking the square root recovers the realizability condition~\eqref{equ:nonblock_simple}. Hence the kaleidocycle is realizable by Theorem~\ref{theorem:nonblocking}.
\end{proof}

\begin{theorem}\label{theorem:noninterference}
	A kaleidocycle composed of three isosceles tetrahedra
	with edge lengths $a$, $b$, and $c$ and their three
	mirror images is realizable with the hinges located
	along edges of length $a$ if and only if
	\begin{equation}\label{eq:noninterference}
		\val(a,b,c):=-5a^6+2a^4(b^2+c^2)+11a^2(b^2-c^2)^2-8(b^2-c^2)^2(b^2+c^2)\geq0.
	\end{equation}
	The inequality can be satisfied only if the hinge
	edge $a$ is strictly shorter than the other two edges.
	The case $\val(a,b,c)=0$ corresponds to a
	kaleidocycle in which, in the extremal configurations,
	adjacent tetrahedra are in face contact.
\end{theorem}

\begin{proof}
	Using the notation of Lemma~\ref{lemma:uv}, we can write the polynomial in Inequality~\eqref{eq:noninterference} as
	\[
	\val(a,b,c)=a^6\bigl((2-8u^2)v-5+11u^2\bigr).
	\]
	Since $a>0$, the condition $\val(a,b,c)\geq0$ is equivalent to
	\[
	(2-8u^2)v\geq5-11u^2.
	\]
	Since $u=\cos\alpha$ by Equation~\eqref{equ:skew_angles} and $\alpha\in(0,\pi/2]$, we have $0\leq u<1$. Moreover, since the faces of a disphenoid are acute triangles, we have $b^2+c^2>a^2$ and hence $v>1$.
	
	For $0\leq u<1/2$, the above inequality is equivalent to the realizability condition of Lemma~\ref{lemma:uv}.
	
	For $1/2\leq u<1$, we have $2-8u^2\leq0$. Since $v>1$,
	\[
	(2-8u^2)v-5+11u^2
	\leq (2-8u^2)-5+11u^2
	=-3(1-u^2)<0.
	\]
	Thus, $\val(a,b,c)\geq0$ is impossible in this range, which establishes the asserted equivalence.
	
	It remains to show that the hinge edge $a$ is strictly shorter than the other two edges $b$ and $c$. By Lemma~\ref{lemma:uv},
	\[
	v\geq\frac{5-11u^2}{2-8u^2}.
	\]
	Moreover,
	\[
	\frac{5-11u^2}{2-8u^2}-(2+u)
	=\frac{(1+u)(8u^2-3u+1)}{2-8u^2}>0,
	\]
	because
	\[
	8u^2-3u+1=8\left(u-\frac{3}{16}\right)^2+\frac{23}{32}>0.
	\]
	Consequently, $v>2+u$, or equivalently,
	\[
	b^2+c^2>2a^2+|b^2-c^2|.
	\]
	Using the identity
	\[
	b^2+c^2-|b^2-c^2|=2\min(b^2,c^2),
	\]
	we obtain
	\[
	2\min(b^2,c^2)>2a^2,
	\]
	which proves that $a<b$ and $a<c$.
	
	Finally, equality in \eqref{eq:noninterference} is equivalent to equality in \eqref{equ:uv_bound} for the lower bound on $v$, and hence, by the derivation in the proof of Lemma~\ref{lemma:uv}, to equality in \eqref{equ:nonblock_simple}. Thus, adjacent tetrahedra reach face contact in the extremal configurations.
\end{proof}

\begin{figure}
	\centering
	
	\begin{minipage}[t]{0.31\textwidth}
		\centering
		\includegraphics[width=\linewidth]{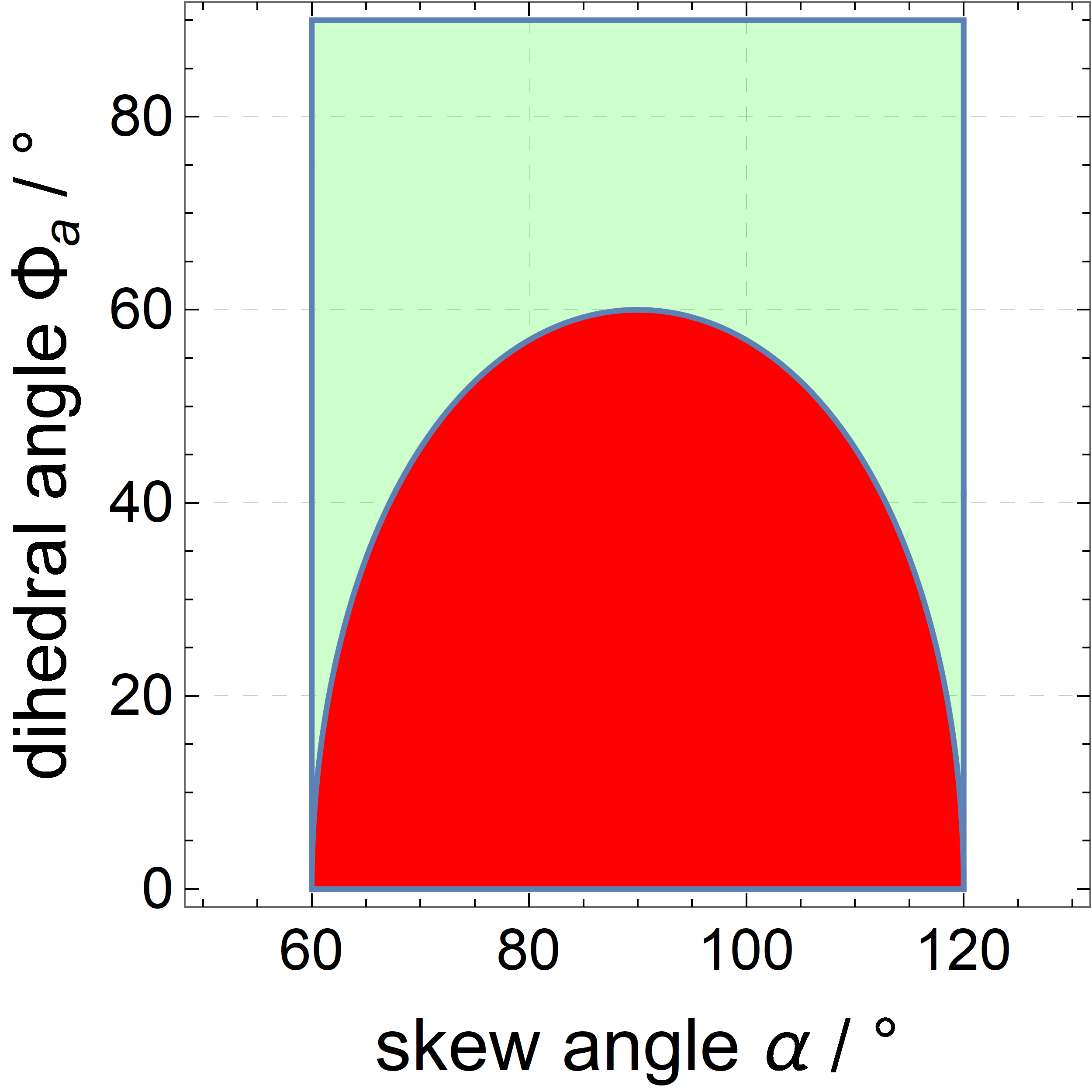}
		
		\vspace{4pt}
		(a) Admissible region of Theorem~\ref{theorem:nonblocking}.
	\end{minipage}
	\hfill
	\begin{minipage}[t]{0.31\textwidth}
		\centering
		\includegraphics[width=\linewidth]{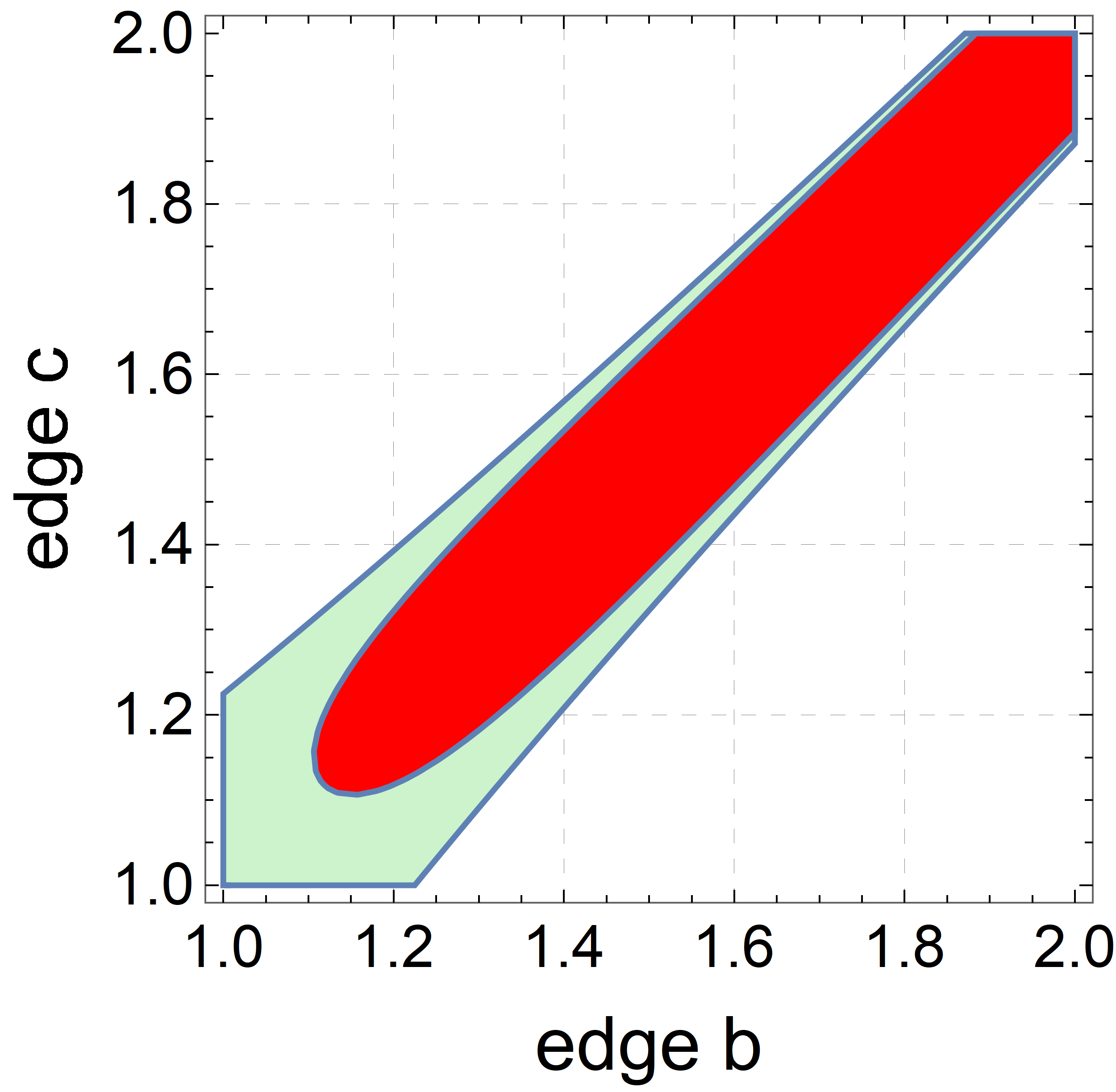}
		
		\vspace{4pt}
		(b) Admissible region of Theorem~\ref{theorem:noninterference}.
	\end{minipage}
	\hfill
	\begin{minipage}[t]{0.31\textwidth}
		\centering
		\includegraphics[width=\linewidth]{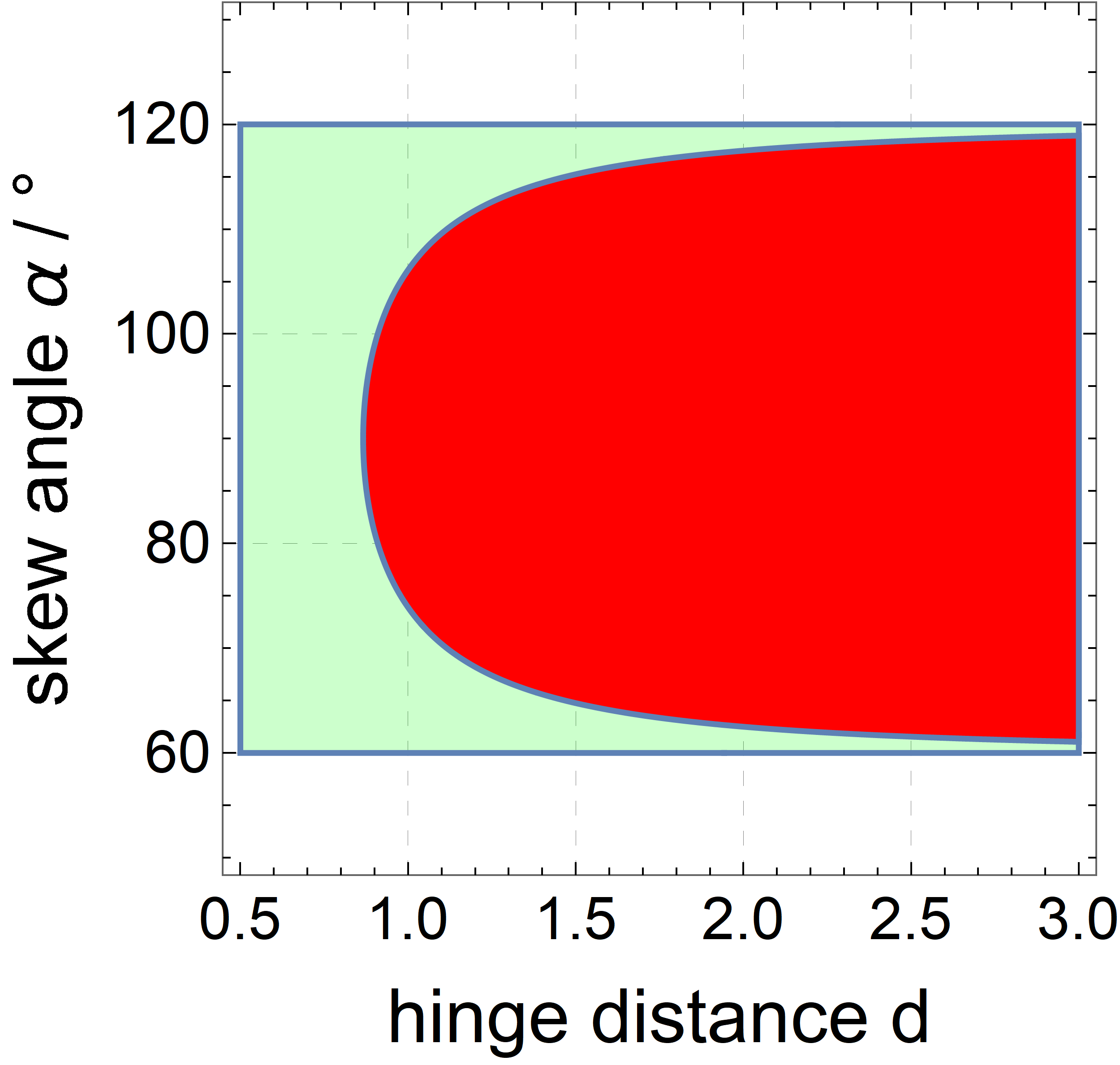}
		
		\vspace{4pt}
		(c) Admissible region of Theorem~\ref{theorem:noninterference2}.
	\end{minipage}

	\caption{Realizability criteria for kaleidocycles. In (b) and (c), the hinge length is normalized to $a=1$.
	 In each diagram, the green region corresponds to the necessary closure condition $\pi/3<\alpha<2\pi/3$, see~\eqref{equ:alpharange}, while the red region indicates the subset satisfying the corresponding realizability criterion.}
	\label{fig:a=1}
	
\end{figure}

The geometric realizability criterion of Theorem~\ref{theorem:nonblocking} can be expressed directly in terms of the hinge length $a$, the distance $d$ between the hinge edges, and their skew angle $\alpha$.

\begin{theorem}\label{theorem:noninterference2}
A kaleidocycle composed of three isosceles tetrahedra with hinge edge length
$a$, skew angle $\alpha$, and hinge distance $d$, together with their
three mirror images, is realizable if and only if
	\begin{align}\label{eq:noninterference2}
	0 < a \le 2d\sqrt{\frac{1}{3}-\cot^2\alpha}.
	\end{align}
	 Equality corresponds to a kaleidocycle in which,
	in the extremal configurations, adjacent tetrahedra are in face contact.
\end{theorem}

\begin{proof}
	By Theorem~\ref{theorem:nonblocking}, the kaleidocycle is realizable if and only if Inequality~\eqref{equ:nonblock_simple} holds.
	
	Using the expression for $\cos\Phi_a$ from Theorem~\ref{theo:dihedral2} and the half-angle identity, we obtain
	\[
	\cos^2\frac{\Phi_a}{2}
	=\frac{1+\cos\Phi_a}{2}
	=\frac12\left(1+\frac{4d^2-a^2\sin^2\alpha}{4d^2+a^2\sin^2\alpha}\right)
	=\frac{4d^2}{4d^2+a^2\sin^2\alpha}.
	\]
	Since $0<\Phi_a<\pi$, taking the positive square root gives
	\[
	\cos\frac{\Phi_a}{2}=\frac{2d}{\sqrt{4d^2+a^2\sin^2\alpha}}.
	\]
	Substituting this into Inequality~\eqref{equ:nonblock_simple} yields
	\[
	\frac{2d\sin\alpha}{\sqrt{4d^2+a^2\sin^2\alpha}}\geq\frac{\sqrt3}{2}.
	\]
	Squaring and rearranging gives
	\[
	a^2\leq4d^2\left(\frac43-\csc^2\alpha\right)
	=4d^2\left(\frac13-\cot^2\alpha\right).
	\]
	Since $a>0$ and $d>0$, this is equivalent to the asserted inequality, and all transformations are reversible.
	
	Finally, equality corresponds to equality in~\eqref{equ:nonblock_simple}, and hence to face contact between adjacent tetrahedra in the extremal configurations.
\end{proof}

\begin{remark}
Theorem~\ref{theorem:noninterference2} is formulated entirely in terms of
parameters that are also defined for the pure linkage shown in
Figure~\ref{fig:linkagebig}. A hinge length outside the admissible range would
cause the hinges to intersect, provided each link bisects its hinge at right
angles.
\end{remark}

\paragraph{Illustrative examples.}

We illustrate the realizability criteria of Theorems~\ref{theorem:nonblocking}--\ref{theorem:noninterference2} by considering three examples: the classical kaleidocycle of Walker, the nearly realizable kaleidocycle investigated in Section~\ref{sec:doecahedron}, and a realizable kaleidocycle with positive clearance, composed of disphenoids with edge lengths $a=1$, $b=1.5$, and $c=1.6$. Their geometric parameters and the numerical evaluation of all three realizability criteria are listed in Table~\ref{tab:realizability_examples}.

\begin{table}[pos=htbp]
	\centering
	\caption{Geometric parameters and evaluation of the realizability criteria for three kaleidocycles with normalized hinge length $a=1$: (a) the classical kaleidocycle of Walker, (b) the kaleidocycle investigated in Section~\ref{sec:doecahedron}, and (c) a realizable kaleidocycle with $b=1.5$ and $c=1.6$. Equation numbers in the first column refer to the formulas used to compute the corresponding quantities or to evaluate the realizability criteria.}
	\label{tab:realizability_examples}
	\begin{tabular}{lccc}
		\hline
		Quantity & (a) & (b) & (c) \\
		\hline
		Hinge length $a$ & $1$ & $1$ & $1$ \\
		Edge length $b$ & $1.11803$ & $1.26185$ & $1.5$ \\
		Edge length $c$ & $1.11803$ & $1.39631$ & $1.6$ \\
		Hinge dihedral angle $\Phi_a$, \eqref{equ:dihedral}
		& $60^\circ$ & $45^\circ$ & $38.009^\circ$ \\
		Skew angle $\alpha$, \eqref{equ:skew_angles}
		& $90^\circ$ & $69.059^\circ$ & $71.941^\circ$ \\
		Hinge distance $d$, \eqref{equ:hinge_distances}
		& $0.86603$ & $1.12738$ & $1.38022$ \\
		\hline
		Maximum required hinge angle, \eqref{equ:thetainterval}
		& $120^\circ$ & $136.027^\circ$ & $131.260^\circ$ \\
		Maximum admissible hinge angle $\pi-\Phi_a$
		& $120^\circ$ & $135^\circ$ & $141.991^\circ$ \\
		Angular clearance & $0^\circ$ & $-1.027^\circ$ & $10.731^\circ$ \\
		\hline
		\multicolumn{4}{l}{\textit{Realizability criteria}} \\[3pt]
		$\sin\alpha\cos(\Phi_a/2)\geq\sqrt{3}/2\approx0.866025$, \eqref{equ:nonblock_simple}
		& $0.866025$ & $0.862856$ & $0.898915$ \\[3pt]
		$\val(a,b,c)\geq0$, \eqref{eq:noninterference}
		& $0$ & $-0.130539$ & $1.97917$ \\[3pt]
		$2d\sqrt{1/3-\cot^2\alpha}\geq a=1$, \eqref{eq:noninterference2}
		& $1$ & $0.974738$ & $1.31524$ \\
		\hline
		Realizable & Yes & No & Yes \\
		\hline
	\end{tabular}
\end{table}

Figure~\ref{fig:compatibility_examples} shows the corresponding compatibility paths together with the admissible angular domains defined by
\[
|\theta|\leq\pi-\Phi_a,\qquad
|\varphi|\leq\pi-\Phi_a.
\]

The three examples illustrate the three possible cases. The classical kaleidocycle of Walker satisfies all three realizability criteria with equality and is therefore realizable, with adjacent tetrahedra coming into face contact in the four extremal configurations. The nearly realizable kaleidocycle of Section~\ref{sec:doecahedron} violates all three criteria and therefore is not realizable under the rigid-body assumption. By contrast, the third kaleidocycle satisfies all three realizability criteria with strict inequality and is therefore realizable without face contact between adjacent tetrahedra throughout the entire motion.

\begin{figure}[pos=htbp]
	\centering
	\setlength{\tabcolsep}{4pt}
	\begin{tabular}{@{}ccc@{}}
		\includegraphics[width=0.32\textwidth]{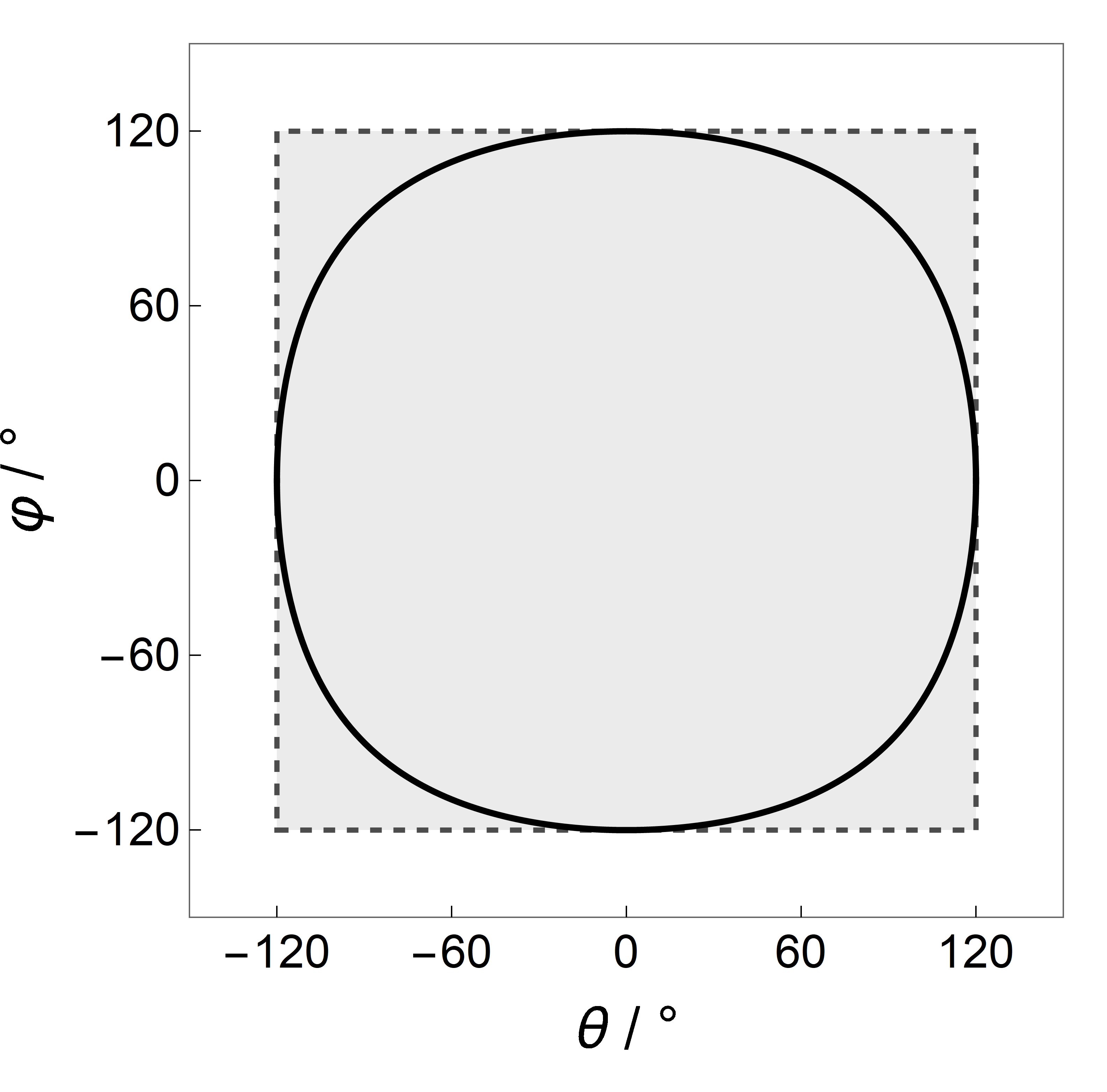} &
		\includegraphics[width=0.32\textwidth]{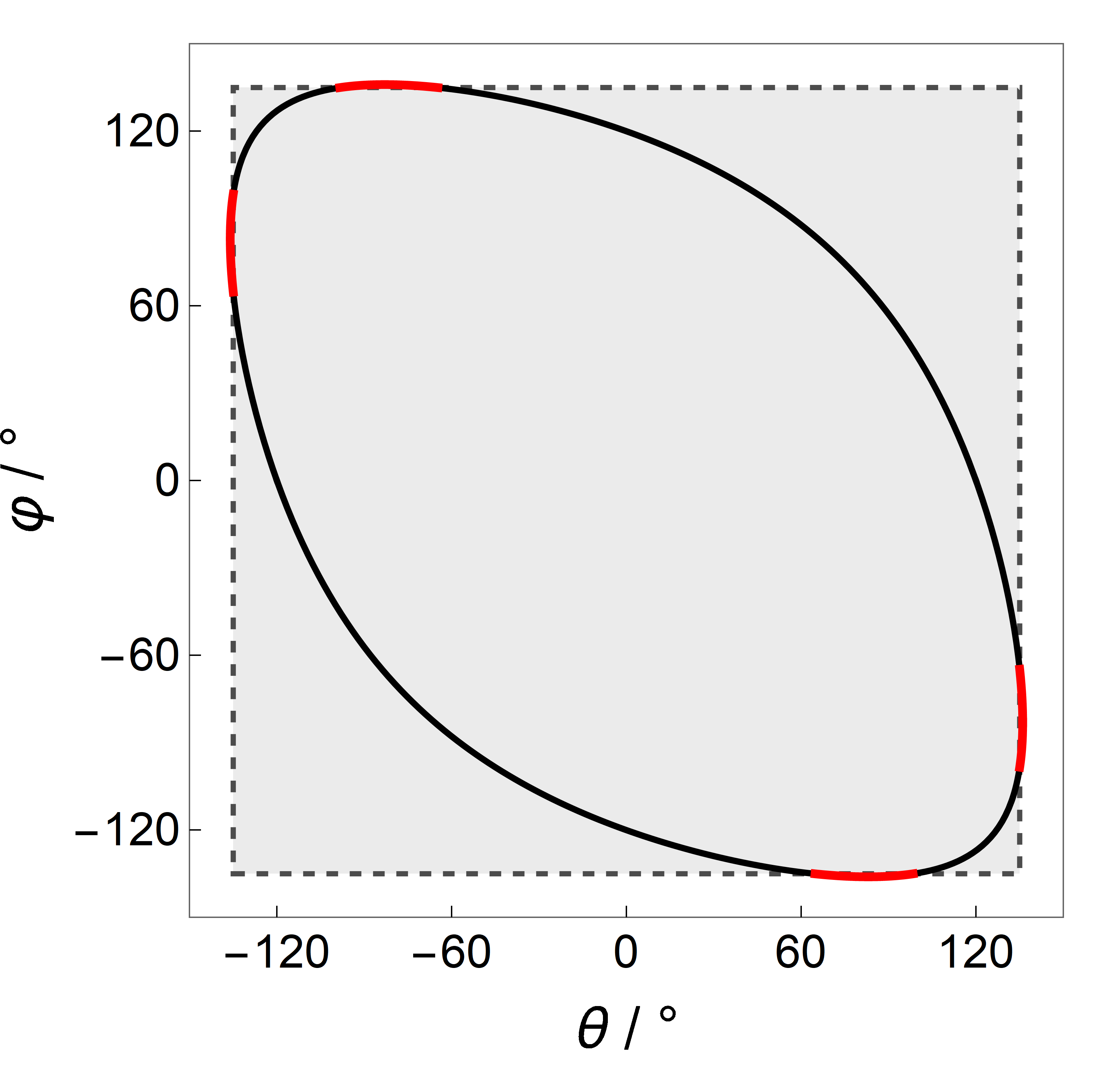} &
		\includegraphics[width=0.32\textwidth]{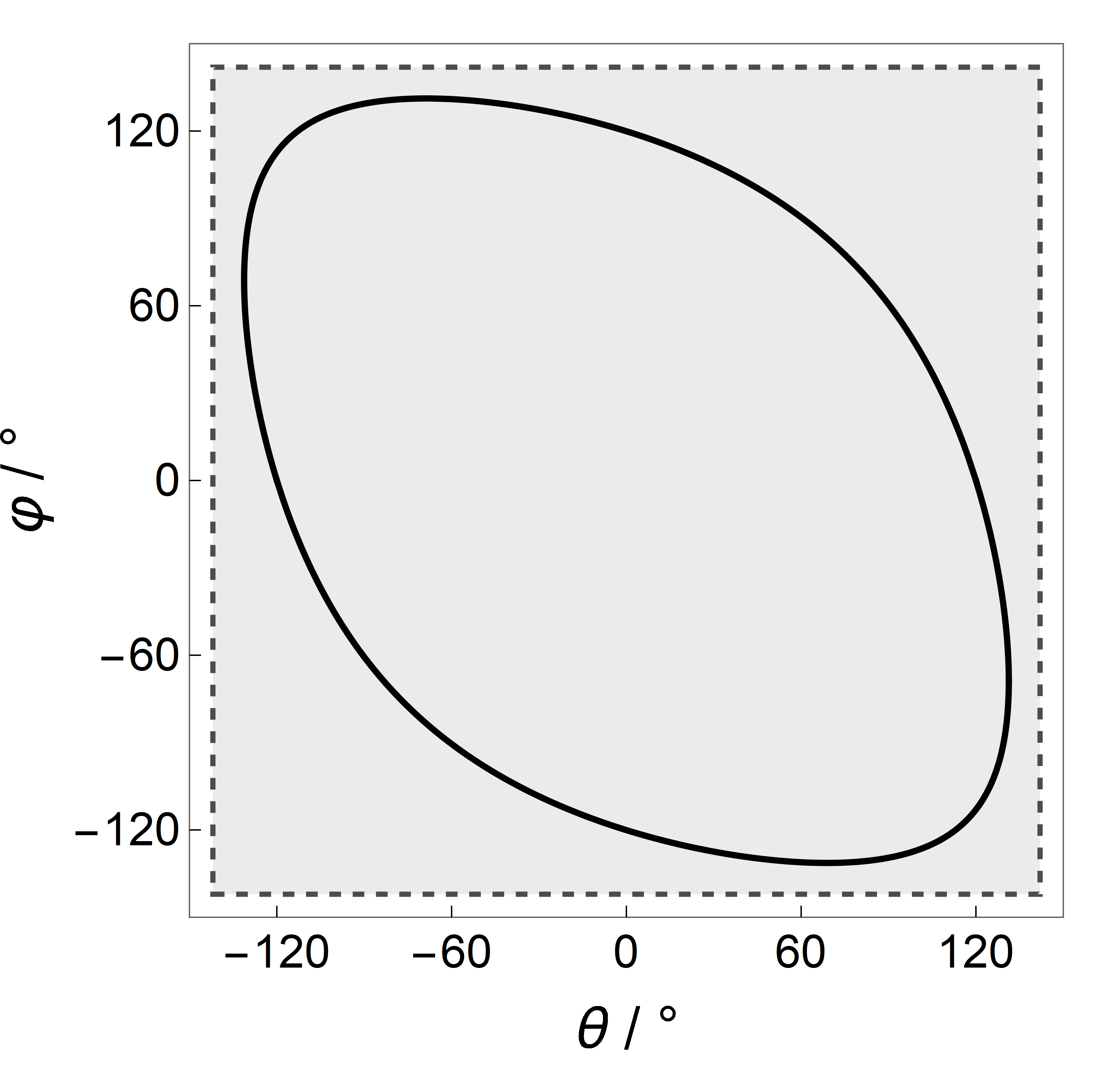} \\
		(a) & (b) & (c)
	\end{tabular}
	\caption{Compatibility paths for (a) the classical kaleidocycle of Walker, (b) the nearly realizable kaleidocycle investigated in Section~\ref{sec:doecahedron}, and (c) a realizable kaleidocycle with positive clearance (\(a=1\), \(b=1.5\), \(c=1.6\)). The dashed lines indicate the admissible angular limits \(\pm(\pi-\Phi_a)\) for both hinge angles. The shaded squares represent the admissible domains, and the red portions of the compatibility path indicate interference between adjacent tetrahedra.}
	\label{fig:compatibility_examples}
\end{figure}

We now examine extremal geometric properties of realizable kaleidocycles. Since the classical kaleidocycle of Walker~\cite{schattschneider1977} plays a distinguished role, we begin with a formal definition.

\begin{definition}\label{def:walker}
	The \emph{classical kaleidocycle of Walker}~\cite{schattschneider1977}
	is the kaleidocycle composed of six congruent isosceles tetrahedra with
	hinge edge length \(a\) and non-hinge edge lengths
	\[
	b=c=\frac{\sqrt5}{2}\,a.
	\]
\end{definition}

Without loss of generality we set $a=1$.
Figure~\ref{fig:a=1}(b) motivates the following result.

\begin{corollary}\label{cor:smallestb=c}
	Let the faces of the tetrahedra be isosceles triangles with $b=c$.
	For hinge length $a=1$, the minimal possible values of $b$ and $c$ for a
	realizable kaleidocycle are
	\[
	b=c=\frac{\sqrt{5}}{2}\approx 1.11803.
	\]
	In this case, adjacent tetrahedra are in face contact in the extremal
	configurations. The corresponding kaleidocycle is, by
	Definition~\ref{def:walker}, the classical kaleidocycle of Walker.
\end{corollary}

\begin{proof}
	For $a=1$ and $b=c$, Inequality~\eqref{eq:noninterference} reduces to
	\[
	4b^2-5\geq0,
	\]
	and hence $b=c\geq\sqrt5/2$. 
	Equality gives the classical kaleidocycle of Walker by Definition~\ref{def:walker}.
	
\end{proof}

This special case agrees with the previously reported restriction for classical six-hinged kaleidocycles. Schönke and Fried~\cite[p.~91, Fig.~2]{Schoenke2019} state that the hinge length must not exceed $2/\sqrt{5}$ times the length of an equal face leg to permit complete eversion without collisions. This is equivalent to $b=c\geq\sqrt{5}a/2$.

The global minimal value for $b$ is slightly smaller.
\begin{corollary}\label{cor:smallest}
For hinge length $a=1$ the edge length $b$ of a realizable kaleidocycle is bounded below. In particular, there exists a global minimum
\begin{equation}\label{equ:bmin}
b_{\min} =\frac{1}{4\sqrt{3}}\sqrt{\sqrt[3]{2 \left(171 \sqrt{57}+751\right)}-\frac{82 \sqrt[3]{4}}{\sqrt[3]{171 \sqrt{57}+751}}+53} \approx 1.10609 
\end{equation}
such that no realizable kaleidocycle exists for $b<b_{\min}$.
In the case $b=b_{\min}$, adjacent tetrahedra are in face contact in the extremal configurations.
By symmetry, an analogous statement holds with $b$ replaced by $c$.
\end{corollary}

\begin{proof}
	Without loss of generality, set $a=1$ and use the notation of Lemma~\ref{lemma:uv}. Since
	\[
	b^2=\frac{v+b^2-c^2}{2}\geq\frac{v-u}{2},
	\]
	Lemma~\ref{lemma:uv} yields
	\begin{equation}\label{equ:bmin_function}
		b^2\geq f(u):=\frac12\left(\frac{5-11u^2}{2-8u^2}-u\right),\qquad 0\leq u<\frac12.
	\end{equation}
	
	Differentiation yields
	\[
	f'(u)=\frac{(1-u)(16u^3+16u^2+8u-1)}{2(1-4u^2)^2}.
	\]
	The polynomial $16u^3+16u^2+8u-1$ is strictly increasing for $u\geq0$, since its derivative $48u^2+32u+8$ is strictly positive in this range. Moreover, the polynomial changes sign between $u=0$ and $u=1/2$. It therefore has exactly one root $u_0$ in $(0,1/2)$. Since $f'(u)<0$ for $0\leq u<u_0$ and $f'(u)>0$ for $u_0<u<1/2$, the function $f$ is strictly decreasing before $u_0$ and strictly increasing thereafter. Hence its unique global minimum on $[0,1/2)$ is attained at $u_0$.
	
	Summarizing the above inequalities, we have
	\[
	b^2\geq\frac{v-u}{2}\geq f(u)\geq f(u_0).
	\]
	Equality between the first and last expressions requires equality in each of the intermediate inequalities. Equality in the first inequality requires $b^2-c^2=-u$, while equality in the last inequality requires $u=u_0$. Thus, the only possible choice attaining the lower bound is
	\[
	b^2=f(u_0),\qquad c^2=f(u_0)+u_0.
	\]
	
	We now verify that these values define a realizable kaleidocycle. Since $0<u_0<1/2$, we have
	\[
	f(u_0)=\frac54+\frac{9u_0^2}{4(1-4u_0^2)}-\frac{u_0}{2}
	>\frac54-\frac14=1.
	\]
	Consequently, $c^2>b^2=f(u_0)>1$, and hence $c>b>1$. Together with $c^2-b^2=u_0<1$, this ensures that the triangular faces are acute and hence define a nondegenerate disphenoid. Moreover,
	\begin{equation}
		v=b^2+c^2=2f(u_0)+u_0=\frac{5-11u_0^2}{2-8u_0^2}.
		\label{equ:bmin_equality}
	\end{equation}
	Thus, the kaleidocycle is realizable by Lemma~\ref{lemma:uv}, and
	\[
	b_{\min}=\sqrt{f(u_0)}
	\]
	is indeed the global minimum.
	
	It remains to determine the exact value of $b_{\min}$. Since $u_0$ is the unique root of the polynomial in the numerator of $f'(u)$ in $(0,1/2)$, we have
	\[
	16u_0^3+16u_0^2+8u_0-1=0.
	\]
	Setting $x=f(u_0)$ and clearing denominators in Equation~\eqref{equ:bmin_function} yields
	\[
	(4-16u_0^2)x=5-11u_0^2-2u_0+8u_0^3.
	\]
	Eliminating $u_0$ from these two equations by taking their resultant yields
	\[
	4096x^3-13568x^2+15856x-6591=0.
	\]
	This cubic has exactly one real root $x_0$. Taking its positive square root yields $b_{\min}=\sqrt{x_0}\approx1.10609$, which is equivalent to the expression in radicals given in~\eqref{equ:bmin}. The symbolic elimination and the equivalence of the two expressions are verified in the supplementary Mathematica notebook.
	
	Finally, by Equation~\eqref{equ:bmin_equality}, the lower bound on $v$ in Lemma~\ref{lemma:uv} is attained with equality. Hence $\val(1,b_{\min},c)=0$, and by Theorem~\ref{theorem:noninterference}, adjacent tetrahedra are in face contact in the extremal configurations. Interchanging $b$ and $c$ gives the corresponding result for $c$.
\end{proof}

The triangular faces of the tetrahedra cannot be arbitrarily oblique. The following corollary establishes a sharp lower bound on the angles between the hinge edge and the adjacent edges.

\begin{corollary}\label{cor:oblique}
	The angles $\angle(a,b)$ and $\angle(a,c)$ between the hinge edge $a$ and the
	adjacent edges $b$ and $c$ of a realizable kaleidocycle are bounded
	below by $60^\circ$. If one of these angles equals $60^\circ$, the skew
	angle~$\alpha$ satisfies
	\begin{equation}\label{equ:quarter}
		\cos\alpha=\frac{1}{4}.
	\end{equation}
	Furthermore, in this extremal case we have $\val(a,b,c)=0$, and the tetrahedra
	are in contact in the extremal configurations.
\end{corollary}

\begin{proof}
	Without loss of generality, set $a=1$ and let $\gamma=\angle(a,b)$. By the law of cosines,
	\[
	\cos\gamma=\frac{1+b^2-c^2}{2b}.
	\]
	Using the notation of Lemma~\ref{lemma:uv}, we introduce the signed difference $w=b^2-c^2$, so that $u=|w|$. By Lemma~\ref{lemma:uv}, realizability implies $-1/2<w<1/2$ and
	\[
	v\geq\frac{5-11w^2}{2-8w^2}.
	\]
	Consequently,
	\[
	\begin{aligned}
		b^2-(1+w)^2
		&=\frac{v+w}{2}-(1+w)^2\\
		&\geq\frac{5-11w^2}{4-16w^2}+\frac{w}{2}-(1+w)^2\\
		&=\frac{(1+w)^2(4w-1)^2}{4(1-4w^2)}\geq0.
	\end{aligned}
	\]
	Since $1+w>0$, we obtain $b\geq1+w$ and hence
	\[
	\cos\gamma=\frac{1+w}{2b}\leq\frac12,
	\]
	which proves that $\gamma\geq60^\circ$.
	
	Equality requires $w=1/4$ and $v=(5-11w^2)/(2-8w^2)=23/8$. These conditions yield
	\[
	b^2=\frac{v+w}{2}=\frac{25}{16},\qquad
	c^2=\frac{v-w}{2}=\frac{21}{16},
	\]
	and hence
	\[
	b=\frac54,\qquad c=\frac{\sqrt{21}}4.
	\]
	These edge lengths define a nondegenerate disphenoid and satisfy the realizability criterion of Lemma~\ref{lemma:uv} with equality. Thus, the lower bound $\gamma=60^\circ$ is attained.
	
	By Equation~\eqref{equ:skew_angle_a}, we have
	\[
	\cos\alpha=|w|=\frac14.
	\]
	Moreover, equality in the realizability criterion implies $\val(1,b,c)=0$, so by Theorem~\ref{theorem:noninterference}, adjacent tetrahedra are in face contact in the extremal configurations.
	
	Interchanging $b$ and $c$ gives the corresponding result for $\angle(a,c)$.
\end{proof}

\begin{remark}
	If $\angle(a,b)=60^\circ$ or $\angle(a,c)=60^\circ$, then in the extremal configurations six triangular faces meet at a common vertex, each with an angle of $60^\circ$. In the first case, the dihedral angle $\Phi_b$ equals $90^\circ$, while in the second case $\Phi_c=90^\circ$. Since adjacent tetrahedra are in face contact, their exterior faces form three coplanar pairs, each spanning an angle of $120^\circ$ at the common vertex. Consequently, all six faces lie in a common plane.
\end{remark}

Figure~\ref{fig:a=1}(c) suggests a lower bound on the distance between opposite hinge edges.

\begin{corollary}
	For the distance $d$ between opposite hinge edges of a realizable kaleidocycle with hinge edge length $a$, we have
	\[
	d\geq\frac{\sqrt3}{2}\,a\approx0.866025\,a.
	\]
	Equality holds only for the classical kaleidocycle of Walker.
\end{corollary}

\begin{proof}
	By Theorem~\ref{theorem:noninterference2},
	\[
	a\leq2d\sqrt{\frac13-\cot^2\alpha}
	\leq\frac{2d}{\sqrt3},
	\]
	which proves the inequality.
	
	Equality requires $\alpha=\pi/2$ and $a=2d/\sqrt3$.
	The first condition implies $b=c$. Using the relation $2d^2=b^2+c^2-a^2$ from Equation~\eqref{equ:hinge_distances}, we obtain $b=c=\sqrt5\,a/2$, corresponding to the classical kaleidocycle of Walker.
\end{proof}

\begin{corollary}
	For the face angle $\delta$ opposite the hinge edge of a realizable kaleidocycle, we have
	\[
	\delta \le \arccos\frac{3}{5}.
	\]
	The unique kaleidocycle with
	\[
	\delta = \arccos\frac{3}{5} \approx 53^\circ
	\]
	is the classical kaleidocycle of Walker.
\end{corollary}

\begin{proof}
	By the law of cosines and Equation~\eqref{equ:hinge_distances}, we have
	\[
	\cos\delta=\frac{b^2+c^2-a^2}{2bc}=\frac{d^2}{bc}.
	\]

	Using $2bc\leq b^2+c^2$, Equation~\eqref{equ:hinge_distances}, and $a^2\leq 4d^2/3$ from the preceding corollary, we obtain
	\[
	2bc\leq b^2+c^2=a^2+2d^2\leq\frac{10}{3}d^2.
	\]

	Consequently,
	\[
	\cos\delta=\frac{d^2}{bc}\geq\frac35,
	\]
	which proves the bound.
	
	Equality requires equality throughout the preceding chain of inequalities, in particular $a^2+2d^2=10d^2/3$, or equivalently $d=\sqrt3\,a/2$. By the preceding corollary, this condition characterizes the classical kaleidocycle of Walker, for which $b=c=\sqrt5 a/2$ and $\cos\delta=3/5$.

\end{proof}

As indicated by Figure~\ref{fig:a=1}(a), the dihedral angle at the hinge edge cannot exceed $60^\circ$.
\begin{corollary}
	For the dihedral angle $\Phi_a$ at the hinge edge of a realizable
	kaleidocycle, we have $\Phi_a \le 60^\circ$. The unique kaleidocycle with
	$\Phi_a = 60^\circ$ is the classical kaleidocycle of Walker.
\end{corollary}

\begin{proof}
	By Theorem~\ref{theorem:nonblocking}, realizability implies
	\[
	\cos\frac{\Phi_a}{2}\geq\frac{\sqrt3}{2\sin\alpha}\geq\frac{\sqrt3}{2}.
	\]
	Hence $\Phi_a\leq60^\circ$.
	
	Equality requires $\alpha=\pi/2$ and equality in~\eqref{equ:nonblock_simple}. By Equation~\eqref{equ:skew_angle_a}, $\alpha=\pi/2$ implies $b=c$. Since equality in~\eqref{equ:nonblock_simple} is equivalent to $\val(a,b,c)=0$ by Theorem~\ref{theorem:noninterference}, we obtain
	\[
	\val(a,b,b)=a^4(4b^2-5a^2)=0.
	\]
	Hence $b=c=\sqrt5\,a/2$, which uniquely identifies the classical kaleidocycle of Walker.
\end{proof}

\begin{remark}
	The preceding corollaries highlight four extremal properties of the
	classical kaleidocycle of Walker. Among realizable kaleidocycles with
	$b=c$, it has the smallest possible non-hinge edge length. Moreover,
	among all realizable kaleidocycles, it uniquely attains the minimum
	distance between opposite hinge edges, the maximal face angle opposite
	the hinge edge, and the maximal hinge dihedral angle.
\end{remark}

\section{The Kaleidocyclic Dodecahedron}\label{sec:doecahedron}

Three-dimensional mobile assemblies based on threefold-symmetric Bricard linkages have previously been investigated by Zhang et al.~\cite{zhang2024assemblies}. While their constructions connect linkages through shared hinges, we investigate face-to-face arrangements of the constituent disphenoids. Experiments with physical models suggested a particular nested arrangement of two congruent kaleidocycles, one in a closed configuration and the other in an open configuration (see Figures~\ref{fig:assembly}(a) and~\ref{fig:assembly}(b)).

The idea is to combine four such nested pairs in such a way that any two nested pairs are in face-to-face contact along a common hinge edge. At each of the six resulting common hinge edges, eight disphenoids would meet, four from each nested pair. This arrangement suggests a polyhedral structure with $48$ triangular faces, as illustrated in Figure~\ref{fig:assembly}(c).

Motivated by this observation, we investigate whether there exists an isosceles tetrahedron for which eight congruent kaleidocycles, consisting in total of $48$ tetrahedra, \emph{fit together perfectly}. More precisely, we mean that every tetrahedron is in face contact with exactly two other tetrahedra --- either from the same kaleidocycle or from a different one --- and that these three tetrahedra share a common hinge edge. We distinguish this local face-contact condition from the additional requirement that the interiors of all $48$ tetrahedra be pairwise disjoint and that their union form a polyhedral body with a closed outer surface.

\begin{lemma}\label{lemma:45degree}
	If four pairs of nested kaleidocycles, each consisting of one closed and one open kaleidocycle, can be assembled without gaps or overlaps to form a polyhedron as illustrated in Figure~\ref{fig:assembly}(c), such that any two nested pairs share exactly one common hinge edge whose interior lies inside the assembled body, then the dihedral angle of each disphenoid at a hinge edge must be $45^\circ$.
	
\end{lemma}

\begin{proof}
	Within each nested pair consisting of one open and one closed kaleidocycle, the two kaleidocycles share three hinge edges, with four tetrahedra meeting at each edge, two from each kaleidocycle. By the assumed arrangement, any two nested pairs share one of these edges, so eight congruent disphenoids meet around it. For chiral disphenoids, the eight tetrahedra alternate in handedness, with four of each type.

	Since the proposed assembly is required to have no gaps or overlaps, the eight disphenoids must fill the entire angular space of $360^\circ$ around any interior point of the common hinge edge. As their dihedral angles are equal, we obtain
	\[
	8\Phi_a=360^\circ,
	\]
	hence $\Phi_a=45^\circ$.
\end{proof}

The shape of a disphenoid is determined by three parameters. If we normalize the hinge edge length to $a=1$ and fix the dihedral angle at this edge to $\Phi_a=45^\circ$, only one free parameter remains. This parameter can in principle be determined using Theorem~\ref{theorem:phi_theta} together with the relations between the hinge angles of the open and closed configurations under the perfect fitting condition. Instead, we will determine it below by a simpler geometric argument based on octahedral symmetry.

\begin{figure}
	\centering
	\begin{minipage}[t]{0.32\textwidth}
		\centering
		\includegraphics[width=\linewidth]{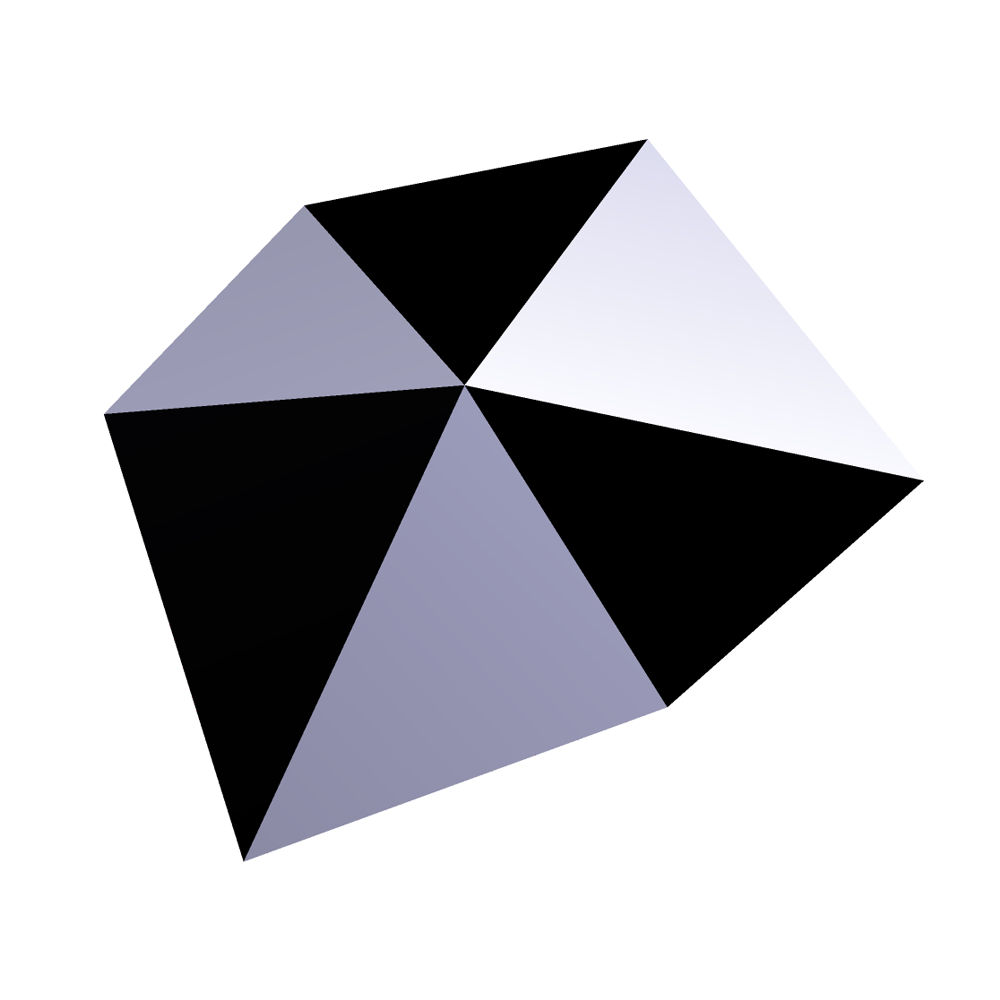}
		\vspace{4pt}
		(a) Closed configuration.
	\end{minipage}
	\hfill
	\begin{minipage}[t]{0.32\textwidth}
		\centering
		\includegraphics[width=\linewidth]{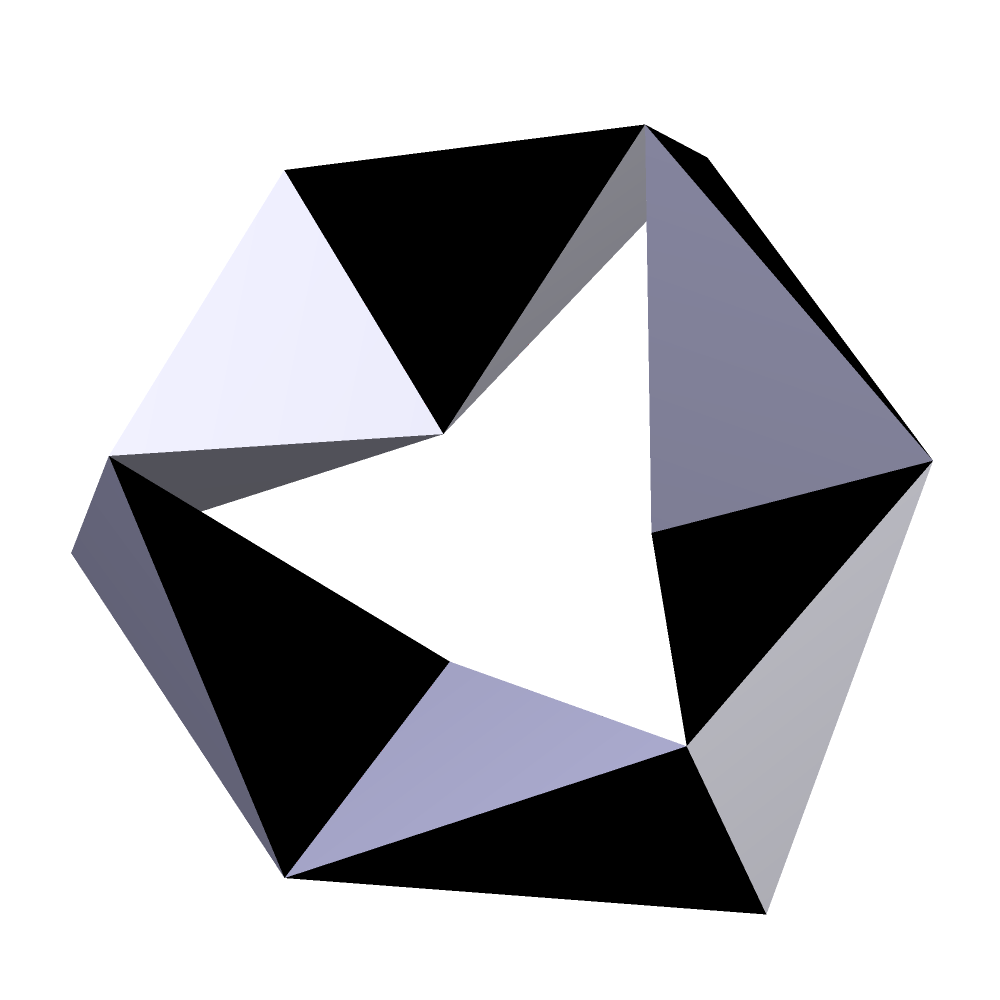}
		\vspace{4pt}
		(b) Open configuration.
	\end{minipage}
	\hfill
	\begin{minipage}[t]{0.32\textwidth}
		\centering
		\includegraphics[width=\linewidth]{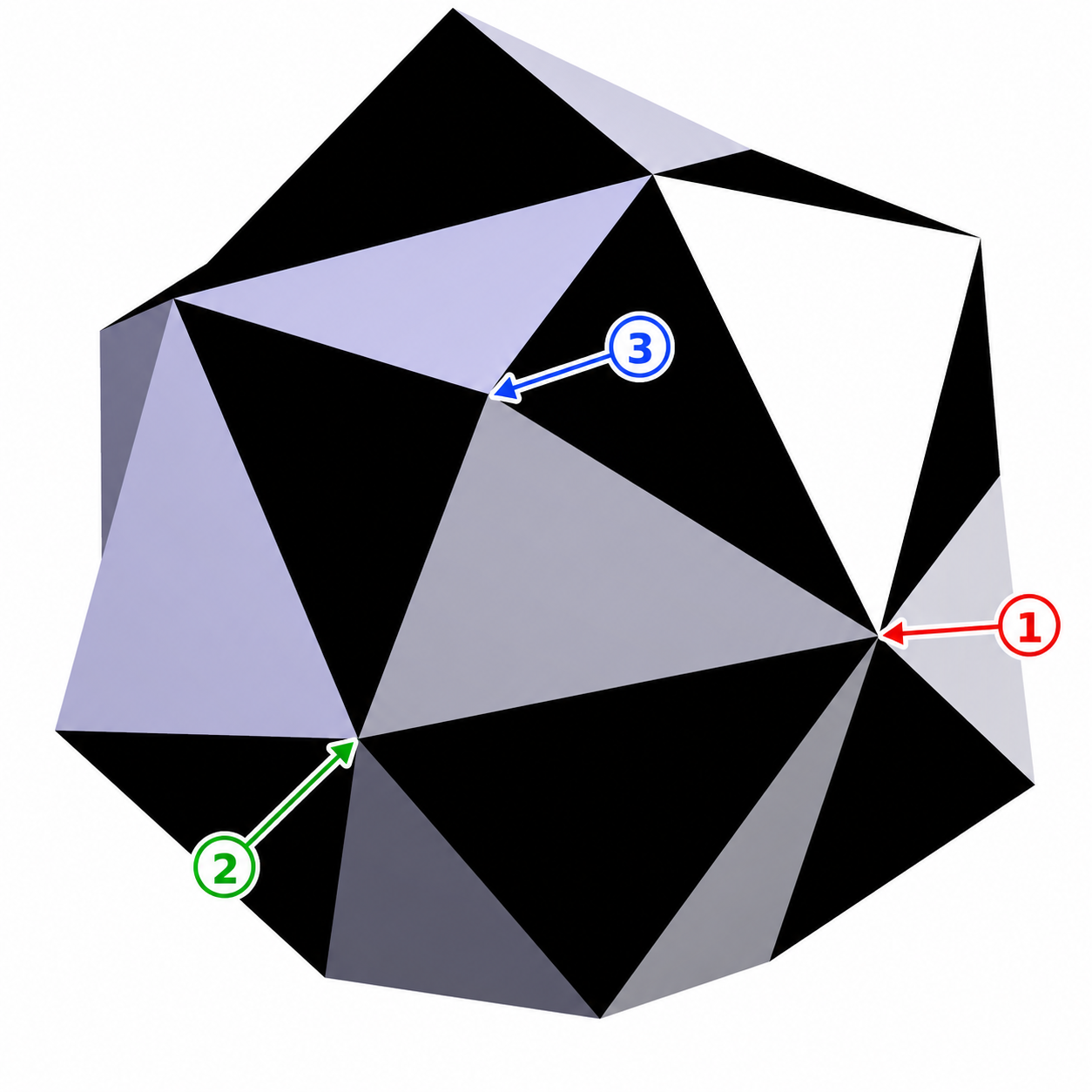}
		\vspace{4pt}
		(c) The kaleidocyclic dodecahedron.
	\end{minipage}
	\caption{Assembly of the kaleidocyclic dodecahedron.}
	\label{fig:assembly}
\end{figure}

The following derivation of the disphenoid dimensions was suggested by a physical model of the kaleidocyclic dodecahedron. By temporarily disregarding the hinge connections and considering only the symmetries of the proposed polyhedron, we obtain a geometric construction based on the full octahedral symmetry group $O_h$.

If the proposed assembly can be realized, its polyhedral surface will exhibit three types of vertices, distinguished by the number and arrangement of the disphenoids meeting at each vertex.

At a vertex of type~1, shown on the right-hand side of Figure~\ref{fig:assembly}(c), eight disphenoids meet along a common (invisible) hinge edge lying on a fourfold symmetry axis.

At a vertex of type~2, shown on the left-hand side of Figure~\ref{fig:assembly}(c), six disphenoids meet. Not all neighboring disphenoids are in face-to-face contact; only those sharing a common vertex of type~1 are. Vertices of type~2 lie on threefold symmetry axes.

At a vertex of type~3, four disphenoids meet, forming two pairs of adjacent tetrahedra in face contact, each pair sharing a common vertex of type~1. Vertices of type~3 lie on twofold symmetry axes.

The arrangement of the three types of vertices on fourfold, threefold, and twofold symmetry axes suggests a construction based on the full octahedral symmetry group $O_h$. This group, which is the symmetry group of a regular octahedron or equivalently a cube, has $48$ elements, comprising $24$ rotations and $24$ orientation-reversing symmetries. With the origin at the center of the cube and the coordinate axes parallel to its edges, the group acts on any point $(x,y,z)$ by permuting the three coordinates and independently changing the sign of each coordinate.

A fundamental region of this group action is given by
\[
0\leq x\leq y\leq z.
\]
It is bounded by the three reflection planes $x=0$, $x=y$, and $y=z$. The $48$ images of this region under $O_h$ cover the entire three-dimensional space and have pairwise disjoint interiors. The pairwise intersections of the bounding planes determine three boundary rays, with direction vectors $(0,0,1)$, $(0,1,1)$, and $(1,1,1)$, lying on fourfold, twofold, and threefold symmetry axes, respectively.

The preceding description of the proposed assembly suggests placing the three vertices of a disphenoid that lie on the outer polyhedral surface on the three boundary rays of the fundamental region. The fourth vertex is placed on the same fourfold symmetry ray as the vertex of type~1, since these two vertices form a hinge edge. We normalize the length of this edge to $a=1$. Since none of the vertices coincides with the origin, the parameters $p$, $q$, and $r$ determining their positions are strictly positive. The resulting arrangement is illustrated in Figure~\ref{fig:posvectors}.

\begin{figure}
	\centering
	\begin{minipage}{0.45\textwidth}
		\includegraphics[width=\linewidth]{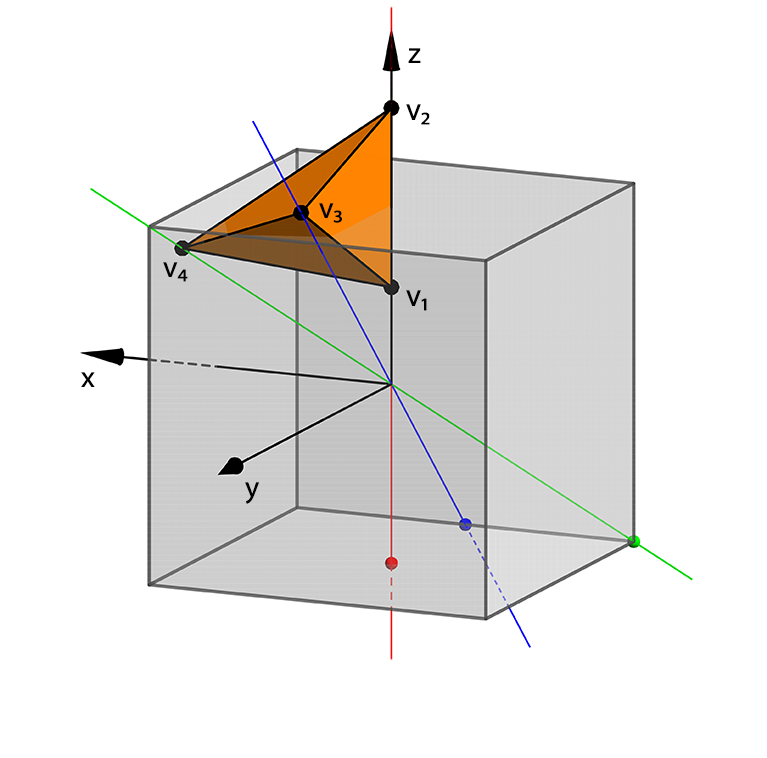}
	\end{minipage}
	\hfill
	\begin{minipage}{0.5\textwidth}
		With cube vertices $(\pm1,\pm1,\pm1)$, the position vectors are:
		\[
		\begin{aligned}
			\mathbf{v}_1 &= p(0,0,1), \quad \text{fourfold axis},\\
			\mathbf{v}_2 &= (p+1)(0,0,1), \quad \text{fourfold axis},\\
			\mathbf{v}_3 &= q(0,1,1), \quad \text{twofold axis},\\
			\mathbf{v}_4 &= r(1,1,1), \quad \text{threefold axis}.
		\end{aligned}
		\]
		The unknowns $p$, $q$, and $r$ satisfy the disphenoid edge conditions:
		\[
		\begin{aligned}
			\|\mathbf{v}_1-\mathbf{v}_2\|
			&=\|\mathbf{v}_3-\mathbf{v}_4\|,\\
			\|\mathbf{v}_1-\mathbf{v}_3\|
			&=\|\mathbf{v}_2-\mathbf{v}_4\|,\\
			\|\mathbf{v}_1-\mathbf{v}_4\|
			&=\|\mathbf{v}_2-\mathbf{v}_3\|.
		\end{aligned}
		\]
	\end{minipage}
	\caption{Position vectors of the disphenoid vertices relative to a cube.}
	\label{fig:posvectors}
\end{figure}

The following theorem determines the dimensions of the disphenoid and establishes the geometric properties of the resulting assembly.

\begin{theorem}[Octahedral construction]
	\label{theorem:octahedral_construction}
There exists a unique disphenoid with vertex positions as in Figure~\ref{fig:posvectors} and $p,q,r>0$. Its hinge dihedral angle is $\Phi_a=45^\circ$. The orbit of this disphenoid under $O_h$ consists of $48$ distinct tetrahedra with pairwise disjoint interiors. Each tetrahedron is in face-to-face contact with two neighboring tetrahedra along faces sharing a common hinge edge of length $a=1$, and eight tetrahedra meet around each such edge on a fourfold symmetry axis.
	
	The outer triangular faces of the $48$ tetrahedra form a closed polyhedral surface with $48$ triangular faces.
\end{theorem}

\begin{proof}
	Using the squared edge lengths, the disphenoid conditions in Figure~\ref{fig:posvectors} reduce to
	\[
	2p=q+r-1,\qquad q^2=2r^2,\qquad r^2+2(q-r)^2=1.
	\]
	Since $q,r>0$, the second equation gives $q=\sqrt{2}\,r$. Substituting this into the third equation and using the first equation, we obtain the unique positive solution
	\[
	r=\frac{1}{\sqrt{7-4\sqrt{2}}},\qquad
	q=\sqrt{2}\,r,\qquad
	p=\frac{q+r-1}{2}.
	\]
	
	Since the fundamental region is convex and contains all four vertices of the disphenoid, it contains the entire disphenoid. Consequently, its $48$ images under $O_h$ are distinct and have pairwise disjoint interiors.
	
	The two faces $\mathbf v_1\mathbf v_2\mathbf v_3$ and $\mathbf v_1\mathbf v_2\mathbf v_4$ lie in the reflection planes $x=0$ and $x=y$, respectively. Reflection in either plane joins the disphenoid face-to-face with a neighboring copy. The angle between these planes within the fundamental region is $45^\circ$, and hence $\Phi_a=45^\circ$. Thus, eight disphenoids meet around each hinge edge, filling the entire angle of $360^\circ$ without gaps.
	
	The three edges of the triangular face $\mathbf v_2\mathbf v_3\mathbf v_4$ lie in the three reflection planes bounding the fundamental region. Since its vertices lie on the three boundary rays, its $48$ images under $O_h$ join edge to edge and form a closed polyhedral surface.
	
	Together, these properties establish that the $48$ disphenoids fit together perfectly and form the claimed polyhedral assembly.
\end{proof}

The exact edge lengths of the disphenoid obtained in Theorem~\ref{theorem:octahedral_construction} are

\begin{equation}
	\label{equ:kalei_length}
	\begin{aligned}
		a &= 1,\\
		b &= \sqrt{\frac{1}{34}\Bigl(15\sqrt{2}-\sqrt{85-34\sqrt{2}}+39\Bigr)}
		\approx 1.26185,\\
		c &= \sqrt{\frac{1}{34}\Bigl(15\sqrt{2}+\sqrt{85-34\sqrt{2}}+39\Bigr)}
		\approx 1.39631.
	\end{aligned}
\end{equation}
The equivalence of these exact expressions to the edge lengths calculated from the vertex coordinates is also verified in the supplementary Mathematica notebook.

The resulting arrangement of $48$ disphenoids admits two different constructions from eight kaleidocycles, as illustrated by the physical models in Figure~\ref{fig:differentways}. These arise from two different ways of pairing the disphenoids along the hinge edges at the vertices of type~1.

In the first construction, four pairs of nested kaleidocycles are assembled, each consisting of one open and one closed kaleidocycle. In the second construction, all eight kaleidocycles are in the closed configuration.

\begin{figure}
	\centering
	\begin{minipage}[t]{0.32\textwidth}
		\centering
		\includegraphics[width=\linewidth]{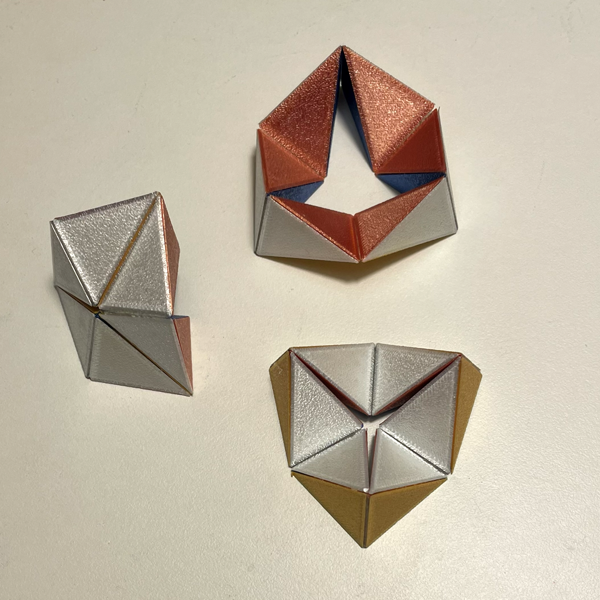}
		\vspace{4pt}
		(a) Some configurations of the same kaleidocycle.
	\end{minipage}
	\hfill
	\begin{minipage}[t]{0.32\textwidth}
		\centering
		\includegraphics[width=\linewidth]{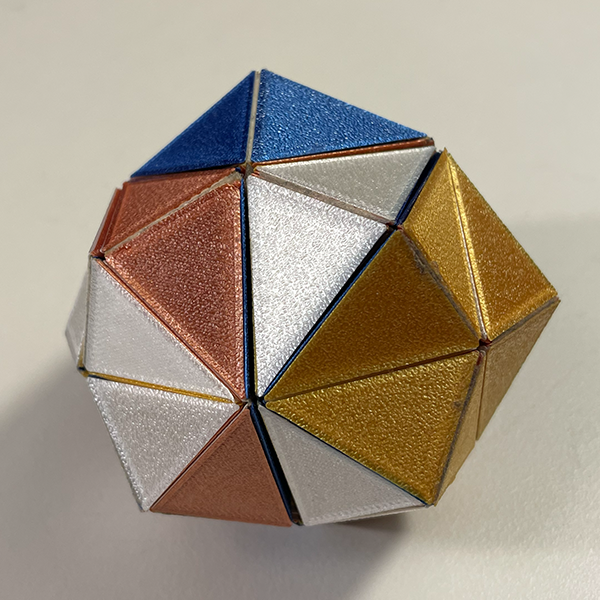}
		\vspace{4pt}
		(b) Assembly from four open and four closed kaleidocycles.
	\end{minipage}
	\hfill
	\begin{minipage}[t]{0.32\textwidth}
		\centering
		\includegraphics[width=\linewidth]{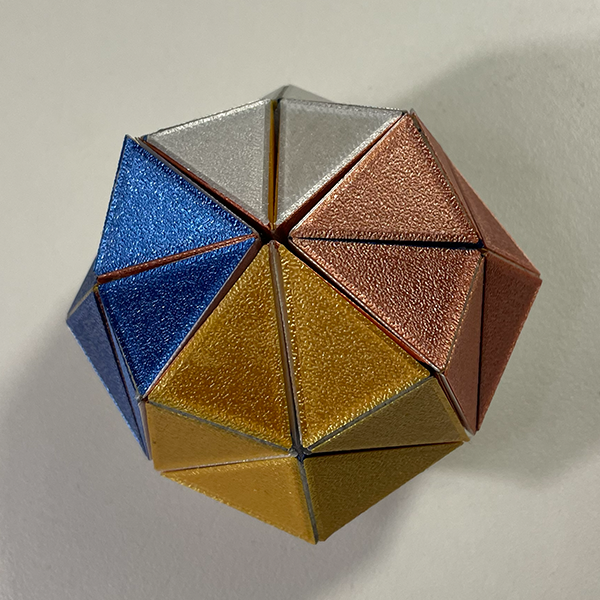}
		\vspace{4pt}
		(c) Assembly from eight closed kaleidocycles.
	\end{minipage}
	\caption{Physical realization of the kaleidocyclic dodecahedron and two different ways to assemble it.}
	\label{fig:differentways}
\end{figure}

Although the open and closed configurations required for these assemblies are attainable, the corresponding kaleidocycle is not realizable, as shown in Table~\ref{tab:realizability_examples}. Its compatibility path extends slightly beyond the admissible angular domain, as illustrated in Figure~\ref{fig:compatibility_examples}(b).

The corresponding skew angle follows from Equation~\eqref{equ:skew_angles} and simplifies to
\[
\alpha=\arctan\!\left(\sqrt{2\sqrt{2}+4}\right)\approx69.059^\circ.
\]

The maximum required hinge angle of approximately $136.027^\circ$ exceeds the maximum admissible hinge angle of $135^\circ$ by only about $1^\circ$. Consequently, at the four extremal configurations, adjacent tetrahedra exhibit a small angular overlap.

We call this kaleidocycle \emph{nearly realizable} because, despite violating the realizability criteria, a physical model is nevertheless fully rotatable. The manufacturing tolerances and slight deformations of the physical components apparently accommodate this small geometric interference, allowing the prototype shown in Figure~\ref{fig:differentways} to complete the entire motion.

\section{Discussion and Conclusions}

In this paper, we derived a new compact closure equation for threefold-symmetric Bricard 6R linkages directly from their threefold symmetry and established its equivalence to the classical closure equation. The resulting formulation facilitates the derivation of the fundamental kinematic relations governing the motion of the corresponding kaleidocycles. These relations form the basis for the necessary and sufficient realizability criteria established in this paper, which in turn yield several extremal geometric properties of realizable kaleidocycles. Finally, we identified a tetrahedral shape that permits a highly symmetric face-to-face assembly of eight congruent kaleidocycles.

The kaleidocyclic dodecahedron illustrates an important distinction between geometric assembly and realizability. Although its constituent disphenoids fit together perfectly in the assembled configuration, the corresponding kaleidocycles do not satisfy the noninterference condition for a complete motion. Nevertheless, physical models permit continuous rotation without observable blocking. This difference can be attributed to the slight overlaps between adjacent tetrahedra predicted by the rigid-body model and the tolerances and deformations of the physical components.

Our realizability criteria apply to threefold-symmetric kaleidocycles composed of six congruent disphenoids with equal hinge lengths, assuming perfectly rigid bodies. 
They provide a direct way to determine whether given disphenoid dimensions permit a complete cycle of motion without blocking and may therefore be useful in the design of kaleidocycle-based mechanisms. Possible extensions include kaleidocycles with unequal hinge lengths or different numbers of tetrahedra, as well as investigations of the effects of flexibility and deformation.

\section*{Acknowledgements}

The author would like to thank the anonymous reviewers for their constructive comments and suggestions. Particular thanks are extended to the reviewer who suggested the more transparent formulation of the realizability criterion now presented in Theorem~\ref{theorem:nonblocking}, which substantially simplified several subsequent mathematical derivations.

\section*{Declaration of generative AI and AI-assisted technologies in the manuscript preparation process}

During the preparation of the original manuscript, the use of ChatGPT (OpenAI) was largely limited to formatting and editorial assistance. The scientific results and mathematical arguments of the original manuscript were developed independently by the author. During the subsequent revision, ChatGPT was used more extensively to assist with scientific discussions, the refinement and simplification of mathematical arguments, symbolic and numerical verification, and the improvement of the presentation. The author reviewed and critically assessed the generated content and takes full responsibility for the scientific results and the content of the publication.

\bibliographystyle{cas-model2-names}
\bibliography{references}

\end{document}